\documentclass[final,3p,times]{elsarticle}

\usepackage{amssymb}
\usepackage{amsthm,bm}
\usepackage{multirow}
\usepackage{amsmath}

\usepackage{color}
\usepackage{hyperref}
\usepackage{lineno}
\usepackage{soul}
\newcommand{\EF}{{\mathrm{eff}}}
\newcommand{\FF}{{\mathrm{ff}}}
\newcommand{\PM}{{\mathrm{pm}}}
\newcommand{\TR}{{\mathrm{tr}}}
\newcommand{\FP}{{\mathrm{fp}}}
\newcommand{\AV}{{\mathrm{ave}}}
\newcommand{\FU}{{\mathrm{+}}}
\newcommand{\RE}{{\mathrm{-}}}

\newtheorem{theorem}{Theorem}
\newtheorem{appendixtheorem}{Theorem}[section]
\renewcommand{\vec}[1]{{\ensuremath{\boldsymbol{\mathrm #1}}}}
\newcommand{\ten}[1]{\ensuremath{\boldsymbol{\mathsf{#1}}}}

\journal{Applied Mathematics and Computation}

\definecolor{Orange}{cmyk}{0.1,0.81,0.87,0}
\definecolor{MidnightBlue}{cmyk}{0.98,0.13,0,0.43}
\definecolor{RoyalBlue}{cmyk}{1,0.50,0,0}
\definecolor{DarkBlue}{cmyk}{1,1,0,0}

\definecolor{Red}{cmyk}{0,1,1,0}
\newcommand{\Rone}[1]{{\color{black}  {#1}}}
\newcommand{\Rthree}[1]{{\color{black}  {#1}}}

\begin{document}

\begin{frontmatter}

\title{Stokes--Brinkman--Darcy models for fluid--porous systems: \\derivation, analysis and validation}

\author[1]{Linheng Ruan}\ead{linheng.ruan@ians.uni-stuttgart.de}

\author[1]{Iryna Rybak\corref{cor2}}\ead{iryna.rybak@ians.uni-stuttagrt.de}

\affiliation[1]{organization={Institute of Applied Analysis and Numerical Simulation, University of Stuttgart},
            addressline={Pfaffenwaldring 57}, 
            city={Stuttgart},
            postcode={70569}, 
            country={Germany}}
\cortext[cor2]{Corresponding author}

\begin{abstract}
Flow interaction between a plain-fluid region in contact with a porous layer attracted significant attention from modelling and analysis sides due to numerous applications in biology, environment and industry. In the most widely used coupled model, fluid flow is described by the Stokes equations in the free-flow domain and Darcy’s law in the porous medium, and complemented by the appropriate interface conditions. However, traditional coupling concepts are restricted, with a few exceptions,  to one-dimensional flows parallel to the fluid--porous interface. 

In this work, we use an alternative approach to model interaction between the plain-fluid domain and porous medium by considering a transition zone, and propose the full- and hybrid-dimensional Stokes--Brinkman--Darcy models. In the first case, the equi-dimensional Brinkman equations are considered in the transition region, and the appropriate interface conditions are set on the top and bottom of the transition zone. In the latter case, we perform a dimensional model reduction by averaging the Brinkman equations in the normal direction and using the proposed transmission conditions. The well-posedness of both coupled problems is proved, and some numerical simulations are carried out in order to validate the concepts.
\end{abstract}




\end{frontmatter}


\biboptions{sort&compress}

\section{Introduction}
\label{sec:intro}

Coupled fluid--porous systems have recently received substantial attention due to their various applications in the biology, environment and industry, such as transport of drugs in living tissues, subsurface drainage, and filtration systems. To avoid high computational complexity by resolving the pore scale geometry, the macroscale perspective is usually considered, where the Stokes equations are applied to describe laminar flow in the free-flow region and Darcy’s law is used to model slow flow through the porous layer. Flow and transport in such systems are highly interface-driven, therefore the choice of interface conditions plays a crucial role. In the last two decades, various coupling strategies have been developed to describe fluid behaviour between the free flow and the porous medium. They can be categorised as a sharp interface or a transition region approach~\cite{Goyeau_Lhuillier_etal_03, Angot_etal_17, Jaeger_etal_01}. There exist a number of works on coupling conditions at the sharp fluid--porous interface for the Stokes--Darcy models that are based on the Beavers--Joseph condition, e.g.~\cite{Beavers_Joseph_67, Discacciati_Quarteroni_09, Jaeger_Mikelic_00, Saffman, Cao_Gunzburger_etal_10, Discacciati_Miglio_Quarteroni_02, Bars_Worster_06, Girault-Riviere-09}.  However, most of them are only suitable for flows parallel to the interface~\cite{Eggenweiler_Rybak_20}. 

Several generalisations to the classical sharp interface model based on the Beavers--Joseph condition~\cite{Beavers_Joseph_67, Saffman} are proposed, many of them based on averaging techniques, e.g., asymptotic modelling, homogenisation, boundary layer theory, and upscaling~\Rone{\cite{Angot_etal_20, Zampogna_Bottaro_16, Naqvi_Bottaro_2021, Ahmed-bottaro-24, Eggenweiler_Rybak_MMS20, Strohbeck-Eggenweiler-Rybak-23, Lacis_Bagheri_17, Jaeger_Mikelic_09, Carraro_etal_15}, including boundary conditions on rough and porous surfaces~\cite{Lacis_etal_20, Marusic-Paloka_Pazanin_2022, Marusic-Paloka_Pazanin_2023}}. Some of these coupling conditions are only theoretically derived and determination of effective parameters appearing in these conditions is still needed to use them in numerical simulations, while others are not verified yet for multi-dimensional flows near the fluid--porous interface. Recently, the generalised coupling conditions are derived using the homogenisation and boundary layer theory~\cite{Eggenweiler_Rybak_MMS20} and rewritten in the same differential form as the Beavers--Joseph condition~\cite{Strohbeck-Eggenweiler-Rybak-23}. The parameters in these conditions are calculated using detailed geometric information on the pore scale and the coupling conditions are validated numerically against pore scale resolved simulations. These conditions extend the classical Beavers--Joseph condition and, in contrast to the original version, accurately capture arbitrary flow scenarios at the fluid--porous interface.

Alternatively, we can consider a narrow transition zone between the plain-fluid domain and the porous medium and derive the dimensionally reduced model~\cite{Angot_etal_17, Angot_etal_20,Jackson_Rybak_etal_12}. This approach has been actively used in the last years to model fluid behaviour in fractured porous media, e.g.~\cite{Martin_etal_05, Knabner_Roberts_14, gander-hennicker-masson-21, Lesinigo_etal_11, brenner-etal-2018, Rybak-Metzger-20, Formaggia_etal_14}. As a result, a hybrid-dimensional model is derived, where the equations in the transition zone are of co-dimension one. In contrast to the sharp interface approximation, dimensional model reduction allows storage and transport of thermodynamic properties in the tangential direction. \Rone{Recently, a coupling framework based on an overlapping domain decomposition method has been proposed in~\cite{Discacciati_Gervasio_24} to simulate the flow in the coupled system where the transition region between the free-fluid and the porous-medium domains is not explicitly defined. Thus, no additional model is introduced therein. This framework is suitable to accurately describe fluid flows in coupled systems with arbitrary flow directions. Another approach that accounts for the double-layer structure near the fluid–porous interface, leading to the Brinkman double-layer model, has been recently developed in~\cite{Kang_Wang_2024} using volume averaging. However, this investigation considers only flows that are parallel to the interface. 
Alternatively, coupled systems with a transition layer between the two flow domains can also be represented using a one-domain approach, as demonstrated in recent studies~\cite{Hernandez-Rodriguez_etal_2022, Arico_etal_2024}.}

In this paper, we propose an efficient alternative to the generalised interface conditions from~\cite{Eggenweiler_Rybak_MMS20}, \Rthree{which are valid for arbitrary flow directions,} and derive the \Rthree{full-dimensional and the} hybrid-dimensional Stokes--Brinkman--Darcy model\Rthree{s} based on our previous work~\cite{Ruan-Rybak-23}. In the full-dimensional model, we couple the Stokes equations in the plain-fluid domain and Darcy's law in the porous medium with the Brinkman equations in the transition region using the appropriate coupling conditions on the top and bottom~\cite{Ruan-Rybak-23, angot2018well}. On the top of the transition region, the interface conditions proposed in~\cite{Angot_etal_17, OchoaTapia_Whitaker_95, ValdesParada_etal_09} are taken into account, including the continuity of velocity and the stress jumps. On the bottom of the transition zone, we consider the classical set of interface conditions based on the Beavers--Joseph--Saffman approach~\cite{Beavers_Joseph_67, Discacciati_Quarteroni_09, Saffman}. The full-dimensional model is applicable for flow problems, where the size of the transition region plays a crucial role.

To derive a hybrid-dimensional model,  we assume a narrow transition zone between the two flow domains. We treat it as a lower-dimensional entity and call it the \emph{complex interface}. This approach extends the asymptotic modelling for multi-dimensional viscous fluid flow and convective transfer at the fluid--porous interface proposed in~\cite{Angot_etal_17}. To develop a dimensionally reduced model, we average the equations in the transition region in the normal direction and make \emph{a priori} hypotheses on the pressure and velocity profiles in a similar way it is done in~\cite{Lesinigo_etal_11, Rybak-Metzger-20} in order to derive the corresponding transmission conditions. We prove the well-posedness for the full-dimensional as well as the  dimensionally reduced Stokes--Brinkman--Darcy models to guarantee the proposed models have unique weak solutions. As a proof of the concept, we provide the numerical convergence study and compare both models with the analytical solutions to validate the models. \Rthree{We also present a comparative study between the full-dimensional and hybrid-dimensional Stokes--Brinkman--Darcy models with respect to accuracy and computational efficiency in the context of a filtration problem.} Thorough investigation of the proposed hybrid-dimensional Stokes--Brinkman--Darcy model and its intercomparison with alternative \Rthree{sharp interface concepts} available in the literature is beyond the scope of this work and is \Rthree{ presented in our recent manuscript~\cite{Ruan_Rybak_24}}.

The paper has the following structure. In section~\ref{sec:model}, we provide the model assumptions considered in this work and present the Stokes, the Brinkman and Darcy's models together with the corresponding interface conditions on the top and on the bottom of the transition region. In section~\ref{sec:model-dimred}, we derive the dimensionally reduced model and propose several transmission conditions required for the model closure. Section~\ref{sec:ana} is devoted to the well-posedness analysis of the coupled full-dimensional and reduced-dimensional models. Numerical results for the developed models are provided in section~\ref{sec:numresults}. Conclusions and future work follow in section~\ref{sec:conclusion}. Some auxiliary results can be found in~\ref{sec:appB}.

\section{Full-dimensional model}
\label{sec:model}
In this section, we present the Stokes, the Brinkman and Darcy's equations that serve as a basis for the full-dimensional formulation as well as for the derivation of the hybrid-dimensional counterpart. 
The entire flow domain $\overline{\Omega} = \overline{\Omega}_\FF \cup \overline{\Omega}_\TR \cup \overline{\Omega}_\PM\subset\mathbb{R}^2$ is composed of the free-flow subdomain $\Omega_\FF \subset\mathbb{R}^2$ and the porous medium $\Omega_\PM\subset\mathbb{R}^2$ divided by the transition region $\Omega_\TR\subset\mathbb{R}^2$ (Fig.~\ref{fig:Fig1intro}, left). The interfaces on the top of the transition zone, $\gamma_\FF=\overline{\Omega}_\FF \cap \overline{\Omega}_\TR \setminus \partial \Omega$, and on the bottom, $\gamma_\PM=\overline{\Omega}_\TR\cap \overline{\Omega}_\PM \setminus \partial \Omega$, are considered to be straight (Fig.~\ref{fig:Fig1intro}, left). We can model a narrow transition region of a thickness $d\in \mathbb{R}^+$ that is significantly smaller than the length of the coupled domain (Fig.~\ref{fig:Fig1intro}) as a complex interface $\gamma$. This interface is a lower-dimensional entity having co-dimension one. In this work, we make the following assumptions: we consider (i)~laminar ($Re \ll 1$), steady-state flow under isothermal conditions;  (ii) the fluid is incompressible with a constant dynamic viscosity $\mu \in \mathbb{R}^+$; (iii)~the transition zone and the porous layer are fully saturated.

Now we introduce the coupled full-dimensional model that includes the Stokes equations describing flow behaviour in the plain-fluid domain, the Brinkman equations in the transition region and Darcy's law in the porous layer. The appropriate interface conditions are chosen on the top and on the bottom of the transition zone, $\gamma_\FF$ and $\gamma_\PM$, respectively. We denote the fluid velocities and pressures as $(\vec{v}_i, p_i)$, $i\in\{\FF, \TR, \PM\}$ within the respective domains. 

\begin{figure}[h!]
    \centering
    \includegraphics[scale= 0.85]{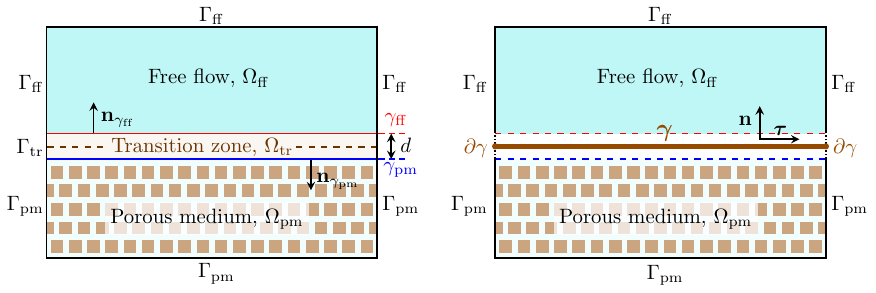}
    \caption{Coupled fluid--porous system with transition zone as full-dimensional inclusion (left) and interface of co-dimension one (right)}
    \label{fig:Fig1intro}
\end{figure}

\subsection{Stokes equations}
\label{sec:stokes}
Under the assumptions at the beginning of section~\ref{sec:model}, the flow in $\Omega_\FF$ is governed by the Stokes equations. The incompressibility condition is ensured through the mass conservation
\begin{equation}
    \nabla \cdot \vec{v}_\FF = 0 \quad \textnormal{in } \Omega_\FF. \label{equ:stokes1}
\end{equation}
For laminar flows, where convective acceleration is disregarded, the momentum conservation equations are derived by applying Newton's law
\begin{equation}
        -\nabla \cdot \ten{T}\left(\vec{v}_\FF, p_\FF\right) = \vec{f}_\FF \quad \textnormal{in } \Omega_\FF.\label{equ:stokes2}
\end{equation}
Here, $\ten{T}\left(\vec{v}_\FF,p_\FF\right) =  \mu \nabla \vec{v}_\FF-p_\FF\ten{I}$ is the viscous stress tensor, and  $\vec{f}_\FF$ is the source term, e.g., body force. On the outer boundary $\Gamma_\FF= \partial \Omega_\FF \setminus \gamma_\FF$, we consider
\begin{equation}
    \vec{v}_\FF = \overline{\vec{v}}_\FF \quad \textnormal{on } \Gamma_{D,\FF} ,\quad  \ten{T}\left(\vec{v}_\FF, p_\FF\right) \cdot \vec{n}_\FF = \overline{\vec{t}}_{\FF} \quad \textnormal{on } \Gamma_{N, \FF} , \label{equ:FFBC}
\end{equation}
where $\overline{\vec{v}}_\FF$, $\overline{\vec{t}}_{\FF}$ are given functions and $\vec{n}_\FF$ is the unit normal vector on the boundary pointing outward from $\Omega_\FF$. Note that there are two separate segments on $\Gamma_{\FF} = \Gamma_{D,\FF} \cup \Gamma_{N,\FF}$, where either the Dirichlet ($D$) or the Neumann ($N$) conditions are set, and $\Gamma_{D,\FF}\ne \emptyset$.

\subsection{Brinkman equations}
\label{sec:brinkman}
We consider the Brinkman equations to describe flow behaviour in the transition zone 
\begin{eqnarray}
    \nabla \cdot \vec{v}_\TR &=& 0 \quad \textnormal{in } \Omega_\TR, \label{equ:Brinkman1}\\
    \mu \ten{K}_\TR^{-1} \vec{v}_\TR-\nabla \cdot  \ten{T}_\EF\left(\vec{v}_\TR, p_\TR\right)  &=& \vec{f}_\TR  \hspace{+1.5ex} \textnormal{in } \Omega_\TR, \label{equ:Brinkman2}
\end{eqnarray}
where $\ten{T}_\EF \left(\vec{v}_\TR,p_\TR\right) = \mu_\EF \nabla \vec{v}_\TR-p_\TR\ten{I}$ represents the stress with the effective viscosity $\mu_\EF \in \mathbb{R}^+$. The transition region permeability $\ten{K}_\TR$ is the second-order tensor which is supposed to be symmetric and positive definite. \Rthree{In this work, we consider homogeneous transition region having constant permeability.} On the outer boundary $\Gamma_\TR= \partial \Omega_\TR \setminus \{ \gamma_\FF \cup \gamma_\PM\} $, the following boundary conditions
\begin{equation}
        \vec{v}_\TR = \overline{\vec{v}}_\TR \quad \textnormal{on } \Gamma_{D,\TR}, \quad \ten{T}_\EF\left(\vec{v}_\TR, p_\TR\right) \cdot \vec{n}_\TR= \overline{\vec{t}}_\TR \quad \textnormal{on } \Gamma_{N,\TR}  \label{equ:TRBC}
\end{equation}
with the given functions $\overline{\vec{v}}_\TR$, $ \overline{\vec{t}}_\TR$ are imposed, where  $\vec{n}_\TR$ is the unit outward normal vector on  $\partial \Omega_\TR$ and $\Gamma_\TR = \Gamma_{D,\TR} \cup \Gamma_{N,\TR}$, $\Gamma_{D,\TR}\cap \Gamma_{N,\TR} = \emptyset$.

\subsection{Darcy's law}
\label{sec:darcy}
To describe slow flow through porous media, Darcy's law is  applied 
\begin{eqnarray}
    \nabla \cdot \vec{v}_\PM &=& q \hspace{+11.7ex}\textnormal{in } \Omega_\PM,\label{equ:Darcy1}  \\
    \vec{v}_\PM &=& -\frac{\ten{K}_\PM}{\mu} \nabla p_\PM   \hspace{2.5ex}\textnormal{in } \Omega_\PM,
\label{equ:Darcy2}\end{eqnarray}
with the symmetric, positive definite and bounded permeability $\ten{K}_\PM$ and the source term $q$. \Rthree{The porous medium has a lower or equal permeability than the transition zone.}
Substituting Eq.~\eqref{equ:Darcy2} into Eq.~\eqref{equ:Darcy1}, we formulate the porous-medium model in the primal form
\begin{eqnarray}
    -\nabla \cdot\left(\frac{\ten{K}_\PM}{\mu} \nabla p_\PM\right) = q. \label{equ:darcyupdated}
\end{eqnarray}
We set the following boundary conditions 
\begin{equation}
     p_\PM = \overline{p}_\PM \quad \textnormal{on } \Gamma_{D,\PM}, \quad \vec{v}_\PM \cdot \vec{n}_\PM = \overline{v}_\PM\quad \textnormal{on } \Gamma_{N,\PM} ,\label{equ:PMBC}
\end{equation}
on the outer boundary $\Gamma_\PM= \partial \Omega_\PM \setminus \gamma_\PM=\Gamma_{D,\PM} \cup \Gamma_{N,\PM}$, 
where $\overline{p}_\PM$ and $ \overline{v}_\PM$ are given functions,  $\vec{n}_\PM$ is the unit outward normal vector on $\partial \Omega_\PM$, and 
$\Gamma_{D,\PM}\cap \Gamma_{N,\PM} = \emptyset$.

\subsection{Interface conditions}
\label{sec:ICfulldim}
To complete the formulation of the coupled full-dimensional model, it is necessary to choose appropriate interface conditions on the interfaces $\gamma_\FF$ and $\gamma_\PM$.
On the top of the transition zone, the continuity of velocity
\begin{equation}
        \vec{v}_\FF = \vec{v}_\TR \quad \textnormal{on } \gamma_\FF,\label{equ:continuityICFF}
\end{equation}
and the stress jump conditions~\cite{Angot_etal_20}:
\begin{equation}
    \ten{T}\left(\vec{v}_\FF, p_\FF\right) \cdot \vec{n}_{\gamma_\FF}-\ten{T}_\EF\left(\vec{v}_\TR, p_\TR\right) \cdot \vec{n}_{\gamma_\FF} = \frac{\mu}{\sqrt{K_\TR}} \vec{\beta} \vec{v}_\FF \quad \textnormal{on }\gamma_\FF \label{equ:stress-jump}
\end{equation}
are taken into account, where $\vec{n}_{\gamma_\FF}$ is the unit normal vector on $\gamma_\FF$ pointing from $\Omega_\TR$ to $\Omega_\FF$ (Fig.~\ref{fig:Fig1intro}, left), $ \vec{\beta}$ denotes the friction tensor and $K_\TR:= \|\ten{K}_\TR \|_{\infty}$. We stress that $\vec{\beta}$ is symmetric and positive semi-definite.

On the bottom of the transition zone, continuity of the normal velocity and the balance of forces are prescribed 
\begin{eqnarray}
        \vec{v}_\TR \cdot \vec{n}_{\gamma_\PM} &=& \vec{v}_\PM \cdot \vec{n}_{\gamma_\PM} \quad \textnormal{on } \gamma_\PM, \label{equ:mass-conservationICPM}\\
        - \vec{n}_{\gamma_\PM} \cdot \ten{T}_\EF\left(\vec{v}_\TR, p_\TR\right) \cdot \vec{n}_{\gamma_\PM}&=& p_\PM \hspace{+7.4ex} \textnormal{on } \gamma_\PM,\label{equ:forcebalanceICPM}
\end{eqnarray}
where the normal vector $\vec{n}_{\gamma_\PM}$ on $\gamma_\PM$ points from $\Omega_\TR$ to $\Omega_\PM$  (Fig.~\ref{fig:Fig1intro}, left). The tangential velocity component is described by the Beavers--Joseph--Saffman condition~\cite{Beavers_Joseph_67,Saffman}:
\begin{equation}
    \vec{v}_\TR\cdot \vec{\tau}_{\gamma_\PM}= - \frac{\sqrt{K_\PM}}{\alpha } \frac{\partial \vec{v}_\TR}{\partial \vec{n}_{\gamma_\PM}}\cdot \vec{\tau}_{\gamma_\PM} \quad \textnormal{on }\gamma_\PM, \label{equ:BJS-ICPM}
\end{equation}
where $\alpha \in \mathbb{R}^+$ is the slip coefficient, $\vec{\tau}_{\gamma_\PM}$ is the unit tangential vector on $\gamma_\PM$ and $K_\PM := \vec{\tau}_{\gamma_\PM} \cdot \ten{K}_\PM\cdot \vec{\tau}_{\gamma_\PM}$.

To summarise, the full-dimensional formulation is composed of the Stokes model~\eqref{equ:stokes1}, \eqref{equ:stokes2} in the plain-fluid subdomain, the equi-dimensional Brinkman equations \eqref{equ:Brinkman1}, \eqref{equ:Brinkman2} in the transition zone, and Darcy's law \eqref{equ:Darcy1}, \eqref{equ:Darcy2} in the porous layer supplemented by the interface conditions \eqref{equ:continuityICFF}--\eqref{equ:BJS-ICPM} on the top and on the bottom of the transition zone. 

\section{Dimensionally reduced model}
\label{sec:model-dimred}
As mentioned in section~\ref{sec:model}, the narrow transition zone can be efficiently modelled as an interface $\gamma$ of co-dimension one (Fig.~\ref{fig:Fig1intro}, right). The transition zone is given by $\displaystyle \Omega_\TR =  \left\{\vec{x} \in \mathbb{R}^2 \big| \, \vec{x} = \vec{s} +1/2 \, \xi \, d(\vec{s}) \vec{n}, \, \vec{s}\in\gamma, \xi \in [-1,1] \right\}$. 
At each point $\vec{s} \in \gamma$, the thickness $d$ is uniquely determined. We can assume sufficiently smooth interface $\gamma$ and thus define the coordinate system with the unit normal $\vec{n}$ and tangential vectors $\vec{\tau}$ (Fig.~\ref{fig:Fig1intro}, right). The corresponding coordinates are denoted by $n$ and $s$, respectively. We note that $\vec{n}_{\gamma_\FF} = \vec{n}$ in Eq.~\eqref{equ:stress-jump}, $\vec{n}_{\gamma_\PM} = -\vec{n}$ in Eqs.~\eqref{equ:mass-conservationICPM}, \eqref{equ:forcebalanceICPM}, and $\vec{\tau}_{\gamma_\PM} = -\vec{\tau}$ in Eq.~\eqref{equ:BJS-ICPM}.  To avoid unnecessary technical difficulties, we restrict ourselves to the case of a constant thickness $d = d(\vec{s})$ and a straight interface $\gamma$ in this work. We will denote the velocity and pressure defined on $\gamma_\FF$ and $\gamma_\PM$ in the reduced model using the same notation as in the full-dimensional formulation.

\subsection{Averaged Brinkman equations}
\label{sec:derivation-dimred}
To derive the hybrid-dimensional model, the Brinkman equations \eqref{equ:Brinkman1}, \eqref{equ:Brinkman2} are projected into the normal and tangential directions on the interface~$\gamma$ and averaged across $\Omega_\TR$. First, we average the mass conservation equation~\eqref{equ:Brinkman1} in the normal direction that yields
\begin{eqnarray}
    \vec{v}_\TR\cdot \vec{n}| _{\gamma_\FF} - \vec{v}_\TR\cdot \vec{n} | _{\gamma_\PM} + d \frac{\partial V_\vec{\tau} }{\partial \vec{\tau}} = 0,\label{equ:integratingMass}
    \qquad V_{\vec{\tau}} := \frac{1}{d}  \int^{d/2}_{-d/2} \vec{v}_\TR \cdot \vec{\tau}\,\mathrm{d} n,
\end{eqnarray}
where $V_{\vec{\tau}}$ is the averaged tangential velocity. Then, we substitute the mass conservation equations~\eqref{equ:continuityICFF}, \eqref{equ:mass-conservationICPM} across the top and bottom boundaries of the transition zone into Eq.~\eqref{equ:integratingMass} and obtain 
\begin{equation}
    \vec{v}_\FF \cdot \vec{n} | _{\gamma_\FF} - \vec{v}_\PM \cdot \vec{n} | _{\gamma_\PM} + d \frac{\partial  V_{\vec{\tau}} }{\partial \vec{\tau}}= 0 \quad \textnormal{on } \gamma. \label{equ:averagedMass}
\end{equation}
Note that in this case we obtain the conservation of mass on the complex interface $\gamma$ with the additional source terms $\vec{v}_\FF \cdot \vec{n} | _{\gamma_\FF}$ and $\vec{v}_\PM \cdot \vec{n} | _{\gamma_\PM}$ coming  from the plain-fluid domain and from the porous layer, respectively, that yields transport of mass along $\gamma$. For the case there exists no transition zone $(d=0)$, we get the classical mass conservation condition across the fluid--porous interface 
\begin{eqnarray*}
        \vec{v}_\FF \cdot \vec{n} | _{\gamma_\FF} - \vec{v}_\PM \cdot \vec{n} | _{\gamma_\PM} = 0.
\end{eqnarray*}

Following the same procedure, we derive the averaged momentum conservation equations. The Brinkman equations \eqref{equ:Brinkman2} are projected into the local orthogonal reference system that yields
\begin{eqnarray}
    \mu\vec{M}_{\vec{n}\vec{n}} \left(\vec{v}_\TR \cdot \vec{n}\right) + \mu\vec{M}_{\vec{n}\vec{\tau}}  \left(\vec{v}_\TR \cdot\vec{\tau}\right) -\left(\nabla\cdot \ten{T}_\EF \left(\vec{v}_\TR, p_\TR\right)\right) \cdot \vec{n} &=& \vec{f}_\TR \cdot \vec{n}, \label{equ:BrinkmanMomentumnormal}\\
    \mu\vec{M}_{\vec{\tau}\vec{n}} \left(\vec{v}_\TR \cdot\vec{n}\right)+ \mu\vec{M}_{\vec{\tau}\vec{\tau}} \left(\vec{v}_\TR \cdot\vec{\tau}\right) -\left(\nabla\cdot \ten{T}_\EF \left(\vec{v}_\TR, p_\TR\right)\right) \cdot \vec{\tau}&=& \vec{f}_\TR \cdot \vec{\tau}, \label{equ:BrinkmanMomentumtangential}
\end{eqnarray}
where $\vec{M}_{\vec{a}\vec{b}} := \vec{a}^\top \ten{K}_\TR^{-1}\vec{b} \in \mathbb{R}$ for the vectors $\vec{a}, \vec{b}\in \mathbb{R}^2$.
Integrating Eq.~\eqref{equ:BrinkmanMomentumnormal} in the normal direction, we get 
\begin{eqnarray}
         -\left(\mu_\EF \frac{\partial \vec{v}_\TR}{\partial \vec{n}}\cdot \vec{n}-p_\TR\right)\Bigg|_{\gamma_\FF}  + \left(\mu_\EF \frac{\partial \vec{v}_\TR}{\partial \vec{n}}\cdot \vec{n} - p_\TR\right)\Bigg|_{\gamma_\PM}  
         = d \left(F_\vec{n} - \mu \vec{M}_{\vec{n} \vec{n}} V_{\vec{n}}- \mu \vec{M}_{\vec{n}\vec{\tau}} V_\vec{\tau} + \mu_\EF \frac{\partial^2 V_\vec{n}}{\partial \vec{\tau}^2}\right), \label{equ:averagedMomentumnormaloriginal}
\end{eqnarray}
where the averaged normal velocity 
$V_{\vec{n}}:= \frac{1}{d}  \int^{d/2}_{-d/2} \vec{v}_\TR \cdot \vec{n}\,\mathrm{d} n $
and the averaged source term $F_{\vec{n}}:= \frac{1}{d}  \int^{d/2}_{-d/2} \vec{f}_\TR\cdot \vec{n}\,\mathrm{d} n $ are defined. 
Taking into account the normal component of the stress jump condition~\eqref{equ:stress-jump} on $\gamma_\FF$ and the balance of normal forces \eqref{equ:forcebalanceICPM} on $\gamma_\PM$, we obtain
\begin{eqnarray}
         \bigg(-\mu \frac{\partial \vec{v}_\FF}{\partial \vec{n}}\cdot \vec{n}+p_\FF +  \frac{\mu }{\sqrt{K_\TR}} \left(\vec{\beta} \vec{v}_\FF\right)  \cdot \vec{n} \bigg)\bigg|_{\gamma_\FF}  - p_\PM \big|_{\gamma_\PM} = d \left(F_{\vec{n}} - \mu\vec{M}_{\vec{n}\vec{n}} V_{\vec{n}} - \mu \vec{M}_{\vec{n}\vec{\tau}} V_{\vec{\tau}} + \mu_\EF \frac{\partial^2 V_{\vec{n}}}{\partial \vec{\tau}^2}\right)\quad \textnormal{on }\gamma.\label{equ:averagedMomentumnormal}
\end{eqnarray}
We note that Eq.~\eqref{equ:averagedMomentumnormal} is the normal component of the momentum conservation equation on the interface $\gamma$. As above, when $d=0$ (no transition zone), we recover from Eq.~\eqref{equ:averagedMomentumnormal} the normal component of the stress jump condition between the free flow and the porous medium~\cite{angot2018well}:
\begin{eqnarray*}
         -\vec{n} \cdot \ten{T} \left(\vec{v}_\FF, p_\FF\right) \cdot \vec{n}\big|_{\gamma_\FF}+ \frac{\mu }{\sqrt{K_\TR}} \left(\vec{\beta} \vec{v}_\FF\right)\cdot\vec{n}  \big|_{\gamma_\FF} &=& p_\PM \big|_{\gamma_\PM} .
\end{eqnarray*}
Analogously, integrating Eq.~\eqref{equ:BrinkmanMomentumtangential} in the normal direction, we have
\begin{eqnarray}
    - \mu_\EF \frac{\partial \vec{v}_\TR}{\partial \vec{n}}\cdot \vec{\tau} \bigg|_{\gamma_\FF} + \mu_\EF \frac{\partial \vec{v}_\TR}{\partial \vec{n}}\cdot \vec{\tau} \bigg|_{\gamma_\PM}
    =  d \left( F_\vec{\tau}- \mu \vec{M}_{\vec{\tau}\vec{n}} V_{\vec{n}} -  \mu \vec{M}_{\vec{\tau}\vec{\tau}} V_\vec{\tau} +  \mu_\EF \frac{\partial^2 V_\vec{\tau}}{\partial \vec{\tau}^2} - \frac{\partial P}{\partial \vec{\tau}} \right).
\label{equ:averagedMomentumtangentialoriginal} \end{eqnarray}
Here, the averaged pressure in the transition region is given by $P:= \frac{1}{d}  \int^{d/2}_{-d/2} p_\TR \,\mathrm{d} n $ and the averaged source term is $F_{\vec{\tau}}:= \frac{1}{d}  \int^{d/2}_{-d/2} \vec{f}_\TR\cdot \vec{\tau}\,\mathrm{d} n$. Considering the tangential component of the stress jump condition~\eqref{equ:stress-jump} on $\gamma_\FF$ and the Beavers--Joseph--Saffman condition \eqref{equ:BJS-ICPM} on $\gamma_\PM$ in Eq.~\eqref{equ:averagedMomentumtangentialoriginal}, we obtain
\begin{eqnarray}
    \left(-\mu \frac{\partial \vec{v}_\FF}{\partial \vec{n}}\cdot \vec{\tau} + \frac{\mu }{\sqrt{K_\TR}} \left(\vec{\beta}\vec{v}_\FF\right)\cdot \vec{\tau}  \right)\Bigg|_{\gamma_\FF} + \frac{\alpha \mu_\EF }{\sqrt{K_\PM}} \vec{v}_\TR \cdot \vec{\tau} |_{\gamma_\PM}
    = d \left( F_{\vec{\tau}}- \mu \vec{M}_{\vec{\tau}\vec{n}} V_{\vec{n}} -  \mu \vec{M}_{\vec{\tau}\vec{\tau}} V_{\vec{\tau}} +  \mu_\EF \frac{\partial^2 V_{\vec{\tau}}}{\partial \vec{\tau}^2} - \frac{\partial P}{\partial \vec{\tau}} \right)\quad \textnormal{on }\gamma.\label{equ:averagedMomentumtangential}   
\end{eqnarray}
Note that Eq.~\eqref{equ:averagedMomentumtangential} is the tangential component of the momentum conservation on the complex interface $\gamma$, where a closure condition to express the tangential velocity $\vec{v}_\TR \cdot \vec{\tau} |_{\gamma_\PM}$ is required (see section~\ref{sec:transcondition}).

\subsection{Transmission conditions}
\label{sec:transcondition}
A closed formulation of the hybrid-dimensional model requires pressure and velocity on the interfaces $\gamma_\FF$ and $\gamma_\PM$ to derive the transmission conditions. Since it is not possible to extrapolate the pressure and velocity values on the interfaces in the reduced-dimensional model, we express $\left(\vec{v}_\TR, p_\TR\right)$ on $\gamma_\FF$ and $\gamma_\PM$ in terms of $\left(\vec{V},P\right)$ and $\left(\vec{v}_i,p_i\right)|_{\gamma_i}$ for $i \in \{\FF,\PM\}$, where the averaged velocity vector is given by $\vec{V}: = V_\vec{n} \vec{n} + V_{\vec{\tau}} \vec{\tau}$.

Following the ideas from \cite{Lesinigo_etal_11, Rybak-Metzger-20}, we make some \emph{a priori} hypotheses on the pressure and velocity profiles in~$\Omega_\TR$. The profiles are given by the interpolation of the interface values on $\gamma_\FF$ and $\gamma_\PM$. In this work, we assume a constant pressure profile 
\begin{eqnarray}
    p_\TR |_{\gamma_\FF} =p_\TR |_{\gamma_\PM} = P. \label{equ:constantpressure} 
\end{eqnarray}
According to~\cite{zimmerman1996hydraulic} a parabolic tangential velocity is expected in the parallel plate model. Similarly, we assume that the tangential velocity has a quadratic profile across the transition zone. Knowing the velocity profile across $\Omega_\TR$, we can express $\vec{v}_\TR \cdot {\vec{\tau}} |_{\gamma_\PM}$ in Eq.~\eqref{equ:averagedMomentumtangential} in terms of $V_{\vec{\tau}}$ and $\vec{v}_\FF \cdot \vec{\tau} |_{\gamma_\FF}$. Following a similar procedure as in \cite[Appendix~A]{Rybak-Metzger-20} and taking into consideration continuity of the tangential component in~\eqref{equ:continuityICFF} and the Beavers--Joseph--Saffman condition~\eqref{equ:BJS-ICPM}, the following closure relations for the tangential velocity are obtained
\begin{eqnarray}
        \vec{v}_\TR \cdot {\vec{\tau}} |_{\gamma_\PM} &=& \frac{2\sqrt{ K_\PM} \left(3V_{\vec{\tau}} - \vec{v}_\FF \cdot \vec{\tau} |_{\gamma_\FF}\right)}{\alpha d+4\sqrt{K_\PM}}, \label{equ:closure_tangential1}\\
        \frac{\partial \vec{v}_{\TR}}{\partial \vec{n}}\cdot\vec{\tau}\bigg|_{\gamma_\FF}&=&\frac{ -6\left(\alpha d + 2\sqrt{K_\PM}\right)V_{\vec{\tau}}+ 4\left(\alpha d + 3\sqrt{K_\PM}\right)\vec{v}_\FF \cdot \vec{\tau}|_{\gamma_\FF}}{d\left(\alpha d+4\sqrt{K_\PM}\right)}\label{equ:closure_tangential2}.
\end{eqnarray}
Substituting Eq.~\eqref{equ:closure_tangential1} into Eq.~\eqref{equ:averagedMomentumtangential}, we complete the expression for the tangential component of momentum conservation on the complex interface $\gamma$:
\begin{eqnarray}
    \left(-\mu \frac{\partial \vec{v}_\FF}{\partial \vec{n}}\cdot \vec{\tau} + \frac{\mu }{\sqrt{K_\TR}} \vec{\tau} \cdot \vec{\beta}\vec{v}_\FF\right)\Bigg|_{\gamma_\FF} +   \frac{ \alpha \mu_\EF \left(6V_{\vec{\tau}} - 2\vec{v}_\FF \cdot \vec{\tau} |_{\gamma_\FF}\right)}{\alpha d+4\sqrt{K_\PM}}
    = d \left( F_{\vec{\tau}}- \mu \vec{M}_{\vec{\tau}\vec{n}} V_{\vec{n}} -  \mu \vec{M}_{\vec{\tau}\vec{\tau}} V_{\vec{\tau}} +  \mu_\EF \frac{\partial^2 V_{\vec{\tau}}}{\partial \vec{\tau}^2} - \frac{\partial P}{\partial \vec{\tau}} \right) \label{equ:averagedMomentumtangentialupdate}.   
\end{eqnarray}

For the normal velocity component, 
\Rthree{it is difficult to justify the profile purely on physical grounds. Therefore, }
we consider different assumptions: linear, piecewise linear, and quadratic profiles. The corresponding closure conditions for these profiles are summarised in Table~\ref{tab:closureconditionnormal}. 

We can rewrite them in a general form as
\begin{eqnarray}
        d \frac{\partial \vec{v}_\TR}{\partial \vec{n}}\cdot\vec{n}\bigg|_{\gamma_\FF} \hspace{+0.6ex}= - \left(\lambda_1 +\lambda_2\right)V_\vec{n}+\lambda_1\vec{v}_\FF \cdot\vec{n}|_{\gamma_\FF}+\lambda_2\vec{v}_\PM\cdot\vec{n}|_{\gamma_\PM},\hspace{-0.25ex}\label{equ:generalnormalclosure-1}\\
        d\frac{\partial \vec{v}_\TR}{\partial \vec{n}}\cdot\vec{n}\bigg|_{\gamma_\PM} =\hspace{+1.8ex}\left(\lambda_1 + \lambda_2\right)V_\vec{n}-\lambda_2\vec{v}_\FF\cdot\vec{n}|_{\gamma_\FF}-\lambda_1\vec{v}_\PM\cdot\vec{n}|_{\gamma_\PM}, \label{equ:generalnormalclosure-2}
\end{eqnarray}
where $\lambda_1>\lambda_2\ge0$ are non-dimensional parameters dependent on the assumption of the normal velocity profile.
\begin{table}[h!]
    \centering
\renewcommand{\arraystretch}{1.2}
\begin{tabular}{c c} 
\hline 
Profiles &  Closure conditions\\
\hline 
$\begin{array}{cc}
     \textnormal{Linear}  \\
     (\lambda_1=2, \lambda_2=0) 
\end{array} $
& \hspace{+4ex}
$\begin{array}{cc}
     \frac{\partial \vec{v}_\TR}{\partial \vec{n}}\cdot\vec{n}\big|_{\gamma_\FF} \hspace{+0.6ex}= -\frac{1}{d}\left(2V_\vec{n}-2\vec{v}_\FF \cdot \vec{n}|_{\gamma_{\FF}} \right) \hspace{+12.3ex} \\
     \frac{\partial \vec{v}_\TR}{\partial \vec{n}}\cdot\vec{n}\big|_{\gamma_\PM}  = \hspace{+1.3ex}\frac{1}{d}\left(2V_\vec{n}- 2\vec{v}_\PM \cdot \vec{n}|_{\gamma_{\PM}}\right)\hspace{+10.7ex}
\end{array}$ \\
\hline
$\begin{array}{cc}
     \textnormal{Piecewise linear}  \\
     (\lambda_1=3,  \lambda_2=1) 
\end{array} $
& \hspace{+4ex}
$\begin{array}{cc}
     \frac{\partial \vec{v}_\TR}{\partial \vec{n}}\cdot\vec{n}\big|_{\gamma_\FF} \hspace{+0.6ex}= -\frac{1}{d}\left(4V_\vec{n}-3\vec{v}_\FF \cdot\vec{n}|_{\gamma_\FF}-\vec{v}_\PM\cdot\vec{n}|_{\gamma_\PM}\right) \hspace{+0.7ex} \\
      \frac{\partial \vec{v}_\TR}{\partial \vec{n}}\cdot\vec{n}\big|_{\gamma_\PM}= \hspace{+1.3ex}\frac{1}{d}\left(4V_\vec{n}-\vec{v}_\FF\cdot\vec{n}|_{\gamma_\FF}-3\vec{v}_\PM\cdot\vec{n}|_{\gamma_\PM}\right)\hspace{+0.7ex}
\end{array}$\\
\hline
$\begin{array}{cc}
     \textnormal{Quadratic}  \\
     (\lambda_1=4,  \lambda_2=2) 
\end{array} $
 & \hspace{+4ex}
 $\begin{array}{cc}
      \frac{\partial \vec{v}_\TR}{\partial \vec{n}}\cdot\vec{n}\big|_{\gamma_\FF}\hspace{+0.6ex} = -\frac{1}{d}\left(6V_\vec{n}-4\vec{v}_\FF \cdot\vec{n}|_{\gamma_\FF}-2\vec{v}_\PM\cdot\vec{n}|_{\gamma_\PM}\right) \hspace{-0.7ex}  \\
      \frac{\partial \vec{v}_\TR}{\partial \vec{n}}\cdot\vec{n}\big|_{\gamma_\PM}  =\hspace{+1.3ex} \frac{1}{d}\left(6V_\vec{n}-2\vec{v}_\FF\cdot\vec{n}|_{\gamma_\FF}-4\vec{v}_\PM\cdot\vec{n}|_{\gamma_\PM}\right) \hspace{-0.8ex}
 \end{array}$\\
\hline
\end{tabular}
    \caption{Closure conditions depending on the normal velocity profile in the transition zone and the corresponding parameters $\lambda_1$ and $\lambda_2$}
    \label{tab:closureconditionnormal}
\end{table}

Substituting Eqs.~\eqref{equ:constantpressure}, \eqref{equ:closure_tangential2} and \eqref{equ:generalnormalclosure-1} into Eq.~\eqref{equ:stress-jump}, we obtain the transmission conditions on~$\gamma_\FF$:
\begin{eqnarray}
    \vec{n}\cdot \ten{T}\left(\vec{v}_\FF, p_\FF \right) \cdot \vec{n} |_{\gamma_\FF} &=& -\frac{\mu_\EF \left(\lambda_1+\lambda_2\right)}{d} V_\vec{n} -P + \frac{\mu_\EF \lambda_1}{d} \vec{v}_\FF\cdot \vec{n} |_{\gamma_\FF}
      + \frac{\mu_\EF \lambda_2 }{d} \vec{v}_\PM\cdot \vec{n}|_{\gamma_\PM}
      +\frac{\mu }{\sqrt{K_{\TR}}} \left(\vec{\beta}\vec{v}_\FF \right)\cdot\vec{n}|_{\gamma_\FF}  , \label{equ:transmissionFFnormal}\\
    \vec{n}\cdot \ten{T}\left(\vec{v}_\FF, p_\FF \right) \cdot \vec{\tau} |_{\gamma_\FF}&=& - \frac{ \mu_\EF\left(6 \alpha d + 12\sqrt{K_\PM}\right) }{d\left(\alpha d+4\sqrt{K_\PM}\right)} V_\vec{\tau} +\frac{  \mu_\EF \left(4 \alpha d + 12\sqrt{K_\PM}\right) }{d\left( \alpha d+4\sqrt{K_\PM} \right)}\vec{v}_\FF \cdot \vec{\tau}|_{\gamma_\FF} + \frac{\mu }{\sqrt{K_\TR}} \left(\vec{\beta}\vec{v}_\FF \right) \cdot \vec{\tau}|_{\gamma_\FF}. \label{equ:transmissionFFtangential}
\end{eqnarray}
Analogously, substitution of Eqs.~\eqref{equ:constantpressure} and \eqref{equ:generalnormalclosure-2} into Eq.~\eqref{equ:forcebalanceICPM} yields the transmission condition on $\gamma_\PM$:
\begin{eqnarray}
    p_\PM |_{\gamma_\PM} &=& -\frac{\mu_\EF}{d} \left(\left(\lambda_1 + \lambda_2\right)V_\vec{n}-\lambda_2\vec{v}_\FF\cdot\vec{n}|_{\gamma_\FF}-\lambda_1\vec{v}_\PM\cdot\vec{n}|_{\gamma_\PM}\right) + P.\hspace{+13ex}
    \label{equ:transmissionPM} 
\end{eqnarray}
To summarise, the proposed hybrid-dimensional model consists of the full-dimensional Stokes equations~\eqref{equ:stokes1}, \eqref{equ:stokes2} in~$\Omega_\FF$, the full-dimensional Darcy law~\eqref{equ:Darcy1}, \eqref{equ:Darcy2} in~$\Omega_\PM$, the dimensionally reduced Brinkman equations \eqref{equ:averagedMass}, \eqref{equ:averagedMomentumnormal},  \eqref{equ:averagedMomentumtangentialupdate} on~$\gamma$ and the transmission conditions \eqref{equ:transmissionFFnormal}--\eqref{equ:transmissionPM}.

\section{Well-posedness of coupled models}
\label{sec:ana}
In this section, we prove existence and uniqueness of weak solutions for the coupled full- and hybrid-dimensional models developed in sections \ref{sec:model} and \ref{sec:model-dimred}.
We assume the permeability tensors $\ten K_\TR$, $\ten K_\PM$ and the friction tensor $\vec{\beta}$ to be uniformly elliptic
\begin{eqnarray}
    k_{\min,i} \| \vec{x} \|^2  \le \vec{x} \cdot \ten{K}_i \cdot \vec{x} \le k_{\max,i} \| \vec{x} \|^2, \quad &\forall \vec{x} \in \mathbb{R}^2, & i\in\{\TR,\PM\},\label{equ:boundedK}\\
    \beta_{\min} \|\vec x \|^2  \le \, \vec x \cdot \vec{\beta} \cdot \vec x  \,\le \beta_{\max} \|\vec x \|^2,\hspace{+0.6ex} \quad  &\forall \vec x \in \mathbb{R}^2,&   \label{equ:boundedbeta}
\end{eqnarray}
where $  k_{\max,i}\ge k_{\min,i}>0$ for $i\in\{\TR,\PM\}$ and $\beta_{\max}\ge\beta_{\min}\ge0.$

\subsection{Weak formulations of the Stokes and Darcy's problems}
First, we introduce the weak formulations of the Stokes and Darcy's problems. In the free flow, we consider test functions from the spaces
$\displaystyle \vec{w}_\FF \in H_\FF : = \left\{\vec{w} \in \left( H^1 \left(\Omega_\FF \right)\right)^2:\, \vec{w}|_{\Gamma_{D,\FF}}=\vec{0}\right\}$ and  $\displaystyle \psi_\FF \in Z_\FF := L^2\left(\Omega_\FF\right)$
equipped with the corresponding norms 
$\displaystyle \left\| \vec{w}_\FF \right\|_{H_\FF}^2 :=  \|\vec{w}_\FF \|^2_{L^2 \left(\Omega_\FF\right)} +  \|\nabla \vec{w}_\FF \|^2_{L^2 (\Omega_\FF)} $ and $\displaystyle  \| \psi_\FF \|_{Z_\FF} := \| \psi_\FF \|_{L^2\left(\Omega_\FF\right)}$, respectively.
We assume that the boundary data in Eq.~\eqref{equ:FFBC} satisfy $\displaystyle \overline{\vec{v}}_\FF \in \left(H^{1/2}\left(\Gamma_{D,\FF}\right)\right)^2$ and $\displaystyle \overline{\vec{t}}_\FF \in \left(H^{-1/2}\left(\Gamma_{N,\FF}\right)\right)^2$. 

Multiplying Eq.~\eqref{equ:stokes2} by the test function $\vec{w}_\FF\in H_\FF$, integrating it over $\Omega_\FF$ by parts and considering the boundary conditions~\eqref{equ:FFBC}, we obtain
\begin{eqnarray}
    \int_{\Omega_\FF} \vec{f}_\FF \cdot \vec{w}_\FF ~\mathrm{d}\vec{x}
           =\int_{\Omega_\FF} \mu \nabla \vec{v}_\FF : \nabla \vec{w}_\FF ~\mathrm{d}\vec{x} - \int_{\Omega_\FF} \! p_\FF \nabla \cdot \vec{w}_\FF~\mathrm{d}\vec{x} 
     +\int_{\gamma_{\FF}} \! \left(\ten{T}\left(\vec{v}_\FF, p_\FF\right) \cdot \vec{n}\right) \cdot \vec{w}_\FF~\mathrm{d} s  
     - \int_{\Gamma_{N,\FF}} \! \overline{\vec{t}}_\FF \cdot  \vec{w}_\FF~\mathrm{d} s. \label{equ:weakformulationofStokesMomentum}
\end{eqnarray}
The weak formulation of the mass conservation \eqref{equ:stokes1} is obtained in the standard way
\begin{eqnarray}
    \int_{\Omega_\FF} \left( \nabla \cdot \vec{v}_\FF\right) \psi_\FF~\mathrm{d}\vec{x} =0,  \quad \forall \psi_\FF \in Z_\FF. \label{equ:weakformulationofStokesmass}
\end{eqnarray}

We consider the porous-medium model in its primal form \eqref{equ:darcyupdated}, and consider the test function space 
$\displaystyle \varphi_\PM \in H_\PM:= \left\{\varphi \in H^1(\Omega_\PM):\, \varphi |_{\Gamma_{D,\PM}}=0 \right\} $
with the norm 
$\displaystyle     \| \varphi_\PM\|_{H_\PM}^2: =  \| \varphi_\PM \|^2_{L^2(\Omega_\PM)} + \| \nabla \varphi_\PM\|^2_{L^2(\Omega_\PM)}$.
The boundary data~\eqref{equ:PMBC} are assumed to be $\overline{p}_\PM \in H^{1/2}(\Gamma_{D,\PM})$ and $\overline{v}_\PM\in H^{-1/2} (\Gamma_{N,\PM})$. Multiplying Eq.~\eqref{equ:darcyupdated} by the test function $\varphi_\PM\in H_\PM$  and integrating over $\Omega_\PM$ by parts, we get 
\begin{eqnarray}
     \int_{\Omega_\PM} q \varphi_\PM ~\mathrm{d}\vec{x} 
       =   \int_{\Omega_\PM} \left( \frac{\ten{K}_\PM}{\mu} \nabla p_\PM\right) \cdot \nabla \varphi_\PM~\mathrm{d}\vec{x}   + \int_{\gamma_\PM} (\vec{v}_\PM \cdot \vec{n}) \varphi_\PM ~\mathrm{d} s +\int_{\Gamma_{N,\PM}} \overline{v}_\PM \varphi_\PM ~\mathrm{d} s\label{equ:weakformulationofDarcy},
\end{eqnarray}
where Darcy's velocity \eqref{equ:Darcy2} and the boundary conditions \eqref{equ:PMBC} are considered. For the model formulation, the stress $\ten{T}\left(\vec{v}_\FF, p_\FF\right) \cdot \vec{n}$ in Eq.~\eqref{equ:weakformulationofStokesMomentum} and the velocity $ \vec{v}_\PM \cdot \vec{n}$ in Eq.~\eqref{equ:weakformulationofDarcy} will be replaced by the interface conditions \eqref{equ:stress-jump} and \eqref{equ:mass-conservationICPM} for the full-dimensional model  and by the transmission conditions \eqref{equ:transmissionFFnormal}--\eqref{equ:transmissionPM} for the reduced-dimensional counterpart, respectively. 

\subsection{Weak formulation of the full-dimensional model}
\label{sec:weak-full}
In this section, we derive the weak form for the full-dimensional Stokes--Brinkman--Darcy problem. For the transition region, we choose test functions from the following spaces $ \vec{w}_\TR \in H_\TR := \left\{ \vec{w} \in \left(H^1 \left(\Omega_\TR\right)\right)^2 : \vec{w}|_{\Gamma_{D,\TR}}= \vec{0} \right\}$ and $ \psi_\TR \in Z_\TR := L^2 ( \Omega_\TR)$
with the norms 
$\displaystyle  \| \vec{w}_\TR \|_{H_\TR}^2 := \|\vec{w}_\TR \|^2_{L^2 (\Omega_\TR)} +  \|\nabla \vec{w}_\TR \|^2_{L^2 (\Omega_\TR)}$ and $\displaystyle  \| \psi_\TR \|_{Z_\TR} := \| \psi_\TR \|_{L^2\left(\Omega_\TR\right)}$.  
The boundary data~\eqref{equ:TRBC} are supposed to be $\overline{\vec{v}}_\TR \in \left(H^{1/2} \left(\Gamma_\TR\right)\right)^2$ and $\overline{\vec t}_\TR \in\left(H^{-1/2} \left(\Gamma_\TR\right)\right)^2 $. 

In a similar manner, we obtain the weak formulation for the Brinkman equations~\eqref{equ:Brinkman1},\eqref{equ:Brinkman2}:
\begin{eqnarray}
        \int_{\Omega_\TR} \left( \nabla \cdot \vec{v}_\TR\right) \psi_\TR~\mathrm{d}\vec{x} =0, \quad \forall \psi_\TR \in Z_\TR,
        \label{equ:weakformulationBrinkmanmass}
\end{eqnarray}
\vspace{-2ex}
\begin{eqnarray}
    \int_{\Omega_\TR} \vec{f}_\TR \cdot \vec{w}_\TR ~\mathrm{d}\vec{x} 
     &=&\int_{\Omega_\TR} \mu \left(\ten{K}_\TR^{-1} \vec{v}_\TR\right)\cdot \vec{w}_\TR ~\mathrm{d}\vec{x}+\int_{\Omega_\TR}   \ten{T}_\EF\left(\vec{v}_\TR, p_\TR\right)  : \nabla \vec{w}_\TR ~\mathrm{d}\vec{x}  -\int_{\gamma_{\FF}} \left(\ten{T}_\EF\left(\vec{v}_\TR, p_\TR\right) \cdot \vec{n}\right) \cdot \vec{w}_\TR~\mathrm{d} s\nonumber\\
     &\quad&+\int_{\gamma_{\PM}} \left(\ten{T}_\EF\left(\vec{v}_\TR, p_\TR\right) \cdot \vec{n}\right)\cdot \vec{w}_\TR~\mathrm{d} s - \int_{\Gamma_{N,\TR}}\overline{\vec{t}}_\TR \cdot \vec{w}_\TR~\mathrm{d} s, \quad \forall \vec{w}_\TR\in H_\TR. \label{equ:wealformulationBrinkmanMomentum}
\end{eqnarray}
Applying the stress jump condition~\eqref{equ:stress-jump}, the balance of normal forces~\eqref{equ:forcebalanceICPM} and the Beavers--Joseph--Saffman condition~\eqref{equ:BJS-ICPM} to the terms 
at the interfaces in Eq.~\eqref{equ:wealformulationBrinkmanMomentum}, we get 
\begin{eqnarray}
     -\int_{\gamma_{\FF}} \left(\ten{T}_\EF\left(\vec{v}_\TR, p_\TR\right) \cdot \vec{n}\right) \cdot \vec{w}_\TR~\mathrm{d} s&=&-\int_{\gamma_{\FF}} \left(\ten{T}\left(\vec{v}_\FF, p_\FF\right) \cdot \vec{n}\right) \cdot \vec{w}_\TR~\mathrm{d} s  +\int_{\gamma_\FF} \frac{\mu}{\sqrt{K_\TR}} \left(\vec{\beta}\vec{v}_\FF\right) \cdot \vec{w}_\TR~\mathrm{d} s, \label{equ:weakformulationTransitionTensorICFF}\\
    \int_{\gamma_{\PM}} \left(\ten{T}_\EF\left(\vec{v}_\TR, p_\TR\right) \cdot \vec{n}\right)\cdot \vec{w}_\TR~\mathrm{d} s  &=& -\int_{\gamma_{\PM}} p_\PM \left(\vec{w}_\TR\cdot\vec{n}\right)~\mathrm{d} s  
   +\int_{\gamma_{\PM}}\frac{  \alpha \mu_\EF }{\sqrt{K_\PM}}\left(\vec{v}_\TR\cdot\vec{\tau}\right)\left(\vec{w}_\TR\cdot\vec{\tau}\right)~\mathrm{d} s. \label{equ:weakformulationTransitionTensorICPM}
\end{eqnarray}
Substituting Eq.~\eqref{equ:mass-conservationICPM} into Eq.~\eqref{equ:weakformulationofDarcy}, we obtain 
\begin{eqnarray}
     \int_{\Omega_\PM} q \varphi_\PM ~\mathrm{d}\vec{x} =  \int_{\Omega_\PM} \left( \frac{\ten{K}_\PM}{\mu} \nabla p_\PM\right) \cdot \nabla \varphi_\PM~\mathrm{d}\vec{x}   
     +\int_{\gamma_\PM} \left(\vec{v}_\TR \cdot \vec{n}\right) \varphi_\PM ~\mathrm{d} s +\int_{\Gamma_{N,\PM}} \overline{v}_\PM \varphi_\PM ~\mathrm{d} s.\label{equ:weakformulationofDarcyupdate}
\end{eqnarray}

For the well-posedness analysis of the coupled full-dimensional problem, we choose the test function spaces 
$\displaystyle \mathcal{H}^\FU := \left\{ ( \vec{w}_\FF, \vec{w}_\TR, \varphi_\PM )  \in H_\FF \times H_\TR\times H_\PM \big| \  \vec{w}_\FF = \vec{w}_\TR \textnormal{ on } \gamma_\FF \right\}$
and $\displaystyle \mathcal{Z}^\FU:= Z_\FF \times Z_\TR$,
and define the corresponding norms 
$\displaystyle \| \vec \theta ^\FU\|_{\mathcal{H}^\FU}^2 :=  \| \vec{w}_\FF \|_{H_\FF}^2  + \| \vec{w}_\TR\|_{H_\TR}^2 + \|\varphi_\PM \|_{H_\PM}^2$ and  $ \displaystyle \| (\psi_\FF, \psi_\TR) \|_{\mathcal{Z}^\FU}^2 :=  \|\psi_\FF\|_{Z_\FF}^2+ \|\psi_\TR\|_{Z_\TR}^2$
for $\vec \theta ^\FU: = (\vec{w}_\FF, \vec{w}_\TR, \varphi_\PM) \in \mathcal{H}^\FU $ and $(\psi_\FF, \psi_\TR) \in \mathcal{Z}^\FU$, respectively. 
Considering Eqs.~\eqref{equ:weakformulationofStokesMomentum}--\eqref{equ:weakformulationofDarcyupdate}, we define the bilinear operators
 $\mathcal{A}^\FU:\mathcal{H}^+\times \mathcal{H}^\FU \to \mathbb{R}$ and $\mathcal{B}^\FU:\mathcal{H}^\FU \times \mathcal{Z}^\FU \to \mathbb{R}$ as
\begin{eqnarray*}
    \mathcal{A}^\FU(\vec{\zeta}^\FU; \vec{\theta}^\FU)  &:=& \mathcal{A}_\FF(\vec{v}_\FF; \vec{w}_\FF)  + \mathcal{A}_\TR (\vec{v}_\TR; \vec{w}_\TR) + \mathcal{A}_\PM (p_\PM; \varphi_\PM) 
    +\mathcal{A}_{\gamma_\FF,\gamma_\PM}(\vec{\zeta}^\FU; \vec{\theta}^\FU)
    \label{equ:defofA^FU}, \\
    \mathcal{B}^\FU(\vec{\zeta}^\FU; \psi_\FF, \psi_\TR) &:=&    -\int_{\Omega_\FF} \left(\nabla \cdot \vec{v}_\FF\right) \psi_\FF ~\mathrm{d}\vec{x}-\int_{\Omega_\TR} \left(\nabla \cdot \vec{v}_\TR\right) \psi_\TR ~\mathrm{d}\vec{x},\label{equ:defofB^FU}
\end{eqnarray*}
for $\vec \zeta ^\FU:=(\vec{v}_\FF, \vec{v}_\TR, p_\PM) \in \mathcal{H}^\FU$, where 
\begin{eqnarray*}
    \mathcal{A}_\FF (\vec{v}_\FF; \vec{w}_\FF) := \int_{\Omega_\FF} \mu \nabla \vec{v}_\FF : \nabla \vec{w}_\FF~\mathrm{d}\vec{x}, \qquad 
    \mathcal{A}_\PM (p_\PM; \varphi_\PM)  := \int_{\Omega_\PM} \left( \frac{\ten{K}_\PM}{\mu} \nabla p_\PM\right) \cdot \nabla \varphi_\PM~\mathrm{d}\vec{x}, \hspace{+23.5ex}
    \\
    \mathcal{A}_\TR (\vec{v}_\TR; \vec{w}_\TR) := \int_{\Omega_\TR} \mu \left(\ten{K}_\TR^{-1} \vec{v}_\TR\right) \cdot \vec{w}_\TR~ \mathrm{d}\vec{x} + \int_{\Omega_\TR} \mu_\EF \nabla \vec{v}_\TR : \nabla\vec{w}_\TR ~\mathrm{d}\vec{x}, \hspace{+46ex}\\
    \mathcal{A}_{\gamma_\FF,\gamma_\PM}(\vec{\zeta}^\FU; \vec{\theta}^\FU) := \int_{\gamma_\FF} \frac{\mu }{\sqrt{{K}_\TR}} (\vec{\beta}\vec{v}_\FF) \cdot \vec{w}_\FF ~\mathrm{d} s-\int_{\gamma_{\PM}} p_\PM (\vec{w}_\TR\cdot \vec{n})~\mathrm{d} s+ \int_{\gamma_\PM} (\vec{v}_\TR \cdot \vec{n}) \varphi_\PM ~\mathrm{d} s   
    + \int_{\gamma_{\PM}}  \frac{\alpha\mu_\EF}{\sqrt{K_\PM}}(\vec{v}_\TR \cdot \vec{\tau})  (\vec{w}_\TR\cdot \vec{\tau})~\mathrm{d} s. 
\end{eqnarray*}
Note that $ \mathcal{A}_\FF$, $\mathcal{A}_\TR$,  and $\mathcal{A}_\PM$ are the bilinear operators in the corresponding flow regions and $\mathcal{A}_{\gamma_\FF,\gamma_\PM}$ represents the bilinear operator on the interfaces. Additionally, the linear functional is defined as 
\begin{eqnarray*}
    \mathcal{L}^\FU\left(\vec \theta^\FU\right) := \int_{\Omega_\FF} \vec{f}_\FF \cdot  \vec{w}_\FF ~\mathrm{d}\vec{x} +\int_{\Omega_\TR} \vec{f}_\TR \cdot \vec{w}_\TR ~\mathrm{d}\vec{x} +\int_{\Omega_\PM}q \varphi_\PM ~\mathrm{d}\vec{x}+ \int_{\Gamma_{N,\FF}}\overline{\vec{t}}_\FF \cdot \vec{w}_\FF~\mathrm{d} s  + \int_{\Gamma_{N,\TR}}\overline{\vec{t}}_\TR \cdot \vec{w}_\TR~\mathrm{d} s -\int_{\Gamma_{N,\PM}} \overline{v}_\PM \varphi_\PM ~\mathrm{d} s.
\end{eqnarray*}
Using the above notations, the weak formulation of the coupled full-dimensional problem \eqref{equ:stokes1}--\eqref{equ:BJS-ICPM} reads: \\
\emph{Find $\vec \zeta ^\FU \in \mathcal{H}^\FU$ and $\left(p_\FF, p_\TR \right)\in \mathcal{Z}^\FU$ such that }
\begin{subequations}
\begin{eqnarray}
    \mathcal{A}^\FU(\vec{\zeta}^\FU; \vec{\theta}^\FU) + \mathcal{B}^\FU(\vec{\theta}^\FU; p_\FF, p_\TR) &=& \mathcal{L}^\FU \left(\vec{\theta}^\FU\right),
    \hspace{+2ex}\forall \vec{\theta}^\FU \in \mathcal{H}^\FU, \label{equ:weakformulationfulldim1}   \\
    \mathcal{B}^\FU (\vec{\zeta}^\FU; \psi_\FF, \psi_\TR) &=& 0, 
    \hspace{+8.2ex}
    \forall \left(\psi_\FF, \psi_\TR\right) \in \mathcal{Z}^\FU.     \label{equ:weakformulationfulldim2}  
\end{eqnarray}
\label{equ:weakformulationfulldim}
\end{subequations}
\emph{Remark 1.} For non-homogeneous Dirichlet boundary data, we have  $(\overline{\vec{v}}_\FF, \overline{\vec{v}}_\TR, \overline{p}_\PM)\in\left(H^{1/2}\left(\Gamma_{D,\FF}\right)\right)^2\times\left(H^{1/2} \left(\Gamma_{D,\TR}\right)\right)^2\times H^{1/2}(\Gamma_{D,\PM})$. The trace operator is surjective. Therefore, we solve~\eqref{equ:weakformulationfulldim} with a lifting $\Xi^\FU \in \left(H^1(\Omega_\FF)\right)^2 \times \left(H^1(\Omega_\TR)\right)^2 \times H^1(\Omega_\PM)$ of $\hat{\vec \zeta}^\FU:=(\hat{\vec{v}}_\FF, \hat{\vec{v}}_\TR, \hat{p}_\PM)\in \left(H^1(\Omega_\FF)\right)^2 \times \left(H^1(\Omega_\TR)\right)^2 \times H^1(\Omega_\PM)$ such that $\hat{\vec \zeta}^\FU - \Xi^\FU = \vec \zeta^\FU \in \mathcal{H}^\FU$.

\subsection{Analysis of the full-dimensional model}
\label{sec:ana-full}
In this section, we establish existence and uniqueness for the full-dimensional Stokes--Brinkman--Darcy model~\eqref{equ:weakformulationfulldim}.
\begin{theorem}[Well-posedness of full-dimensional model] \label{thm:fulldimwellposedness}
    The coupled full-dimensional problem~\eqref{equ:weakformulationfulldim} has a unique solution $\left(\vec{\zeta}^\FU, p_\FF,p_\TR\right)\in\mathcal{H}^\FU\times \mathcal{Z}^\FU$. 
\end{theorem} 
%
\begin{proof}
To prove the well-posedness of \eqref{equ:weakformulationfulldim}, we verify conditions \eqref{continuous1}--\eqref{inf-sup} presented in \ref{sec:appB}.
%
%
To obtain the continuity of $\mathcal{A}^\FU$, we take into account the Cauchy--Schwarz inequality and the uniform ellipticity conditions~\eqref{equ:boundedK}, \eqref{equ:boundedbeta}
that yields
\begin{eqnarray*}
        | \mathcal{A}^\FU(\vec{\zeta}^\FU; \vec{\theta}^\FU)| \le |\mathcal{A}_\FF(\vec{v}_\FF; \vec{w}_\FF) | + |\mathcal{A}_\TR (\vec{v}_\TR; \vec{w}_\TR)| + |\mathcal{A}_\PM (p_\PM; \varphi_\PM)| + |\mathcal{A}_{\gamma_\FF,\gamma_\PM}(\vec{\zeta}^\FU; \vec{\theta}^\FU)|\hspace{+19.4ex} \\
        \le \mu \| \nabla \vec{v}_\FF \|_{L^2(\Omega_\FF)} \| \nabla\vec{w}_\FF \|_{L^2 (\Omega_\FF)} + \frac{\mu}{ k_{\min, \TR}} \| \vec{v}_\TR\|_{L^2( \Omega_\TR)} \|\vec{w}_\TR\|_{L^2( \Omega_\TR)} + \mu_\EF \| \nabla\vec{v}_\TR\|_{L^2(\Omega_\TR)} \|\nabla\vec{w}_\TR\|_{L^2(\Omega_\TR)} \hspace{+3.7ex}\\
         + \frac{k_{\max, \PM}}{\mu} \|\nabla p_\PM\|_{L^2(\Omega_\PM)}\|\nabla \varphi_\PM\|_{L^2(\Omega_\PM)} + \frac{\mu \beta_{\max}}{\sqrt{K_\TR}}  \|\vec{v}_\FF\|_{L^2 (\gamma_\FF)} \|\vec{w}_\FF \|_{L^2(\gamma_\FF)}  + \| p_\PM\|_{L^2(\gamma_\PM)} \|\vec{w}_\TR \|_{L^2(\gamma_\PM)} \hspace{-2ex}\\ 
         + \|\vec{v}_\TR \|_{L^2(\gamma_\PM)} \|\varphi_\PM\|_{L^2(\gamma_\PM)}  + \frac{\alpha \mu_\EF}{\sqrt{K_\PM}}  \|\vec{v}_\TR \|_{L^2(\gamma_\PM)} \|\vec{w}_\TR \|_{L^2(\gamma_\PM)}.\hspace{+29ex}
    \end{eqnarray*}
Applying \eqref{equ:tracegeneral} from the trace theorem on $\gamma_\FF$ and $\gamma_\PM$:
\begin{eqnarray}
    \exists \hspace{+1ex}C_{\gamma_i,i}\hspace{+2ex}>0 \hspace{+0.5ex} \textnormal{ s.t. } \hspace{+2ex} \|\vec{v}_i\|_{L^2(\gamma_i)}\hspace{+1ex} \le C_{\gamma_i,i} \|\vec{v}_i \|_{H_i} \hspace{+4ex} \textnormal{ for } i \in \{\FF, \TR\},\label{equ:traceFFTR}\\
        \exists\hspace{0.5ex} C_{\gamma_\PM,\PM}>0 \hspace{+0.5ex}\textnormal{ s.t. }\hspace{+0.5ex} \|p_\PM \|_{L^2(\gamma_\PM)} \le C_{\gamma_\PM,\PM} \|p_\PM \|_{H_\PM},\label{equ:tracePM} \hspace{+10.3ex}
\end{eqnarray}
and inequality \eqref{equ:quadraticrule}, we estimate further  
    \begin{eqnarray*}
     | \mathcal{A}^\FU(\vec{\zeta}^\FU; \vec{\theta}^\FU)| 
     \le  \left(\mu+\frac{\mu \beta_{\max}}{ \sqrt{K_{\TR}}} C_{\gamma_\FF,\FF}^2 \right) \| \vec{v}_\FF \|_{H_\FF} \|\vec{w}_\FF \|_{H_\FF} \!+ \!\left(\frac{\mu}{ k_{\min, \TR}}+\mu_\EF+\frac{\alpha \mu_\EF}{ \sqrt{K_{\PM}} } C_{\gamma_\PM,\TR}^2\right) \| \vec{v}_\TR\|_{H_\TR} \|\vec{w}_\TR\|_{H_\TR}\hspace{+1.5ex}\\
     + \frac{k_{\max, \PM}}{\mu} \|p_\PM \|_{H_\PM}\|\varphi_\PM\|_{H_\PM} 
     + C_{\gamma_\PM,\TR}C_{\gamma_\PM,\PM} \left(\| p_\PM\|_{H_\PM} \|\vec{w}_\TR \|_{H_\TR} +  \|\vec{v}_\TR\|_{H_\TR} \| \varphi_\PM\|_{H_\PM}\right)  \\
    \le  \frac{C_{\mathcal{A}^\FU}}{4}\left(\| \vec{v}_\FF \|_{H_\FF}+ \| \vec{v}_\TR\|_{H_\TR}+\|p_\PM\|_{H_\PM} \right) \left(\|\vec{w}_\FF \|_{H_\FF}+ \|\vec{w}_\TR\|_{H_\TR} +\|\varphi_\PM\|_{H_\PM}\right)\hspace{+11.4ex}\\
    \le C_{\mathcal{A}^\FU} \| \vec{\zeta}^\FU\|_{\mathcal{H}^\FU}\| \vec{\theta}^\FU\|_{\mathcal{H}^\FU},\hspace{+53.5ex}
\end{eqnarray*}
where 
\begin{eqnarray*}
    C_{\mathcal{A}^\FU}:= 4\max \left\{ \mu\left(1+\frac{\beta_{\max}}{\sqrt{K_\TR}} C_{\gamma_\FF,\FF}^2 \right),\,  \frac{\mu }{k_{\min, \TR}}+\mu_\EF\left(1+\frac{ \alpha }{ \sqrt{K_{\PM}}} C_{\gamma_\PM,\TR}^2\right),\,  \frac{k_{\max, \PM}}{\mu},\, C_{\gamma_\PM,\TR}C_{\gamma_\PM,\PM}\right\}.
\end{eqnarray*}
Thus, the continuity of $\mathcal{A}^\FU$ is ensured.

In a similar way, taking the first expression in \eqref{equ:quadraticrule} into account, we obtain
\begin{eqnarray*}
    | \mathcal{B}^\FU(\vec{\zeta}^\FU; \psi_\FF, \psi_\TR) | &\le& \bigg|\int_{\Omega_\FF} (\nabla \cdot \vec{v}_\FF) \psi_\FF ~\mathrm{d}\vec{x}\bigg| + \bigg| \int_{\Omega_\TR} (\nabla \cdot \vec{v}_\TR) \psi_\TR ~\mathrm{d}\vec{x}\bigg|  
    \le \| \vec{v}_\FF \| _{H_\FF} \| \psi_\FF\|_{Z_\FF} + \|\vec{v}_\TR\|_{H_\TR}\|\psi_\TR\|_{Z_\TR} \\
    &\le&
    \left(\| \vec{v}_\FF\|_{H_\FF}+\| \vec{v}_\TR\|_{H_\TR} \right) \left(\| \psi_\FF\|_{Z_\FF}  +\|\psi_\TR\|_{Z_\TR} \right)\le 2 \| \vec{\zeta}^\FU\|_{\mathcal{H}^\FU} \| (\psi_\FF, \psi_\TR)\|_{\mathcal{Z}^\FU}.
\end{eqnarray*}
Therefore, the continuity of the bilinear form $\mathcal{B}^\FU$ is guaranteed.

Now, we show the coercivity of the bilinear form $\mathcal{A}^\FU$ on $\displaystyle \mathit{Kern}(\mathcal{B}^\FU)=\left\{ \vec{\theta}^\FU \in \mathcal{H}^\FU | \nabla \cdot \vec{w}_\FF = 0,\, \nabla \cdot \vec{w}_\TR=0 \right\}$. Recalling the uniform ellipticity conditions \eqref{equ:boundedK}, \eqref{equ:boundedbeta}, the auxiliary result \eqref{equ:auxiliarypoincare} from the Poincar{\'e} theorem and taking into account the physical parameters $\mu$, $\mu_\EF$, $\alpha$, $\sqrt{K_\TR}$, $\sqrt{K_\PM} > 0$ and $\beta_{\min} \geq 0$, we obtain
\begin{eqnarray*}
        \mathcal{A}^\FU(\vec{\zeta}^\FU; \vec{\zeta}^\FU) = \mathcal{A}_\FF(\vec{v}_\FF; \vec{v}_\FF)  + \mathcal{A}_\TR (\vec{v}_\TR; \vec{v}_\TR) + \mathcal{A}_\PM (p_\PM; p_\PM) +\mathcal{A}_{\gamma_\FF,\gamma_\PM}(\vec{\zeta}^\FU; \vec{\zeta}^\FU)\hspace{+34ex}
        \\
                \ge   \frac{\mu}{\tilde{C}_{P,\FF}} \|\vec{v}_\FF \|_{H_\FF}^2 + \min\bigg\{ \frac{\mu}{ k_{\max, \TR}}, \mu_\EF\bigg\}\| \vec{v}_\TR \|^2_{H_\TR} + \frac{k_{\min, \PM}}{\mu\tilde{C}_{P,\PM}} \|p_\PM\|^2_{H_\PM}+ \frac{\mu \beta_{\min}}{  \sqrt{K_{\TR}}} \|\vec{v}_\FF\|_{L^2(\gamma_\FF)}^2 + \frac{ \alpha \mu_\EF}{\sqrt{K_\PM}}  \|\vec{v}_\TR\cdot \vec{\tau}\|_{L^2(\gamma_\PM)}^2 \\
        \ge  C^*_{\mathcal{A}^\FU} \|\vec{\zeta}^\FU\|^2_{\mathcal{H}^\FU}
   ,\qquad  C^*_{\mathcal{A}^\FU}:= \min \left\{  \frac{\mu}{\tilde{C}_{P,\FF}}, \frac{\mu}{ k_{\max, \TR}}, \mu_\EF,  \frac{k_{\min, \PM}}{\mu\tilde{C}_{P,\PM}}\right\}.\hspace{+36ex}
\end{eqnarray*}
Therefore, the coercivity of $\mathcal{A}^\FU$ is proved. 

In the last step, we show that $\mathcal{B}^\FU$ is inf-sup-stable following the idea from~\cite[Section 7.1.2]{boffi2013mixed}.  We choose $\psi_i \in Z_i$, $i\in\{\FF,\TR\}$ as a source term in the Poisson problem
\begin{eqnarray}
        -\Delta \xi_i = \psi_i \quad \textnormal{in }  \Omega_i,\qquad \textnormal{with}\quad 
        \xi_i = 0 \quad \textnormal{on }  \Gamma_{D,i}, \quad
        \nabla \xi_i \cdot \vec{n}_i = 0 \quad \textnormal{on } \Gamma_{N,i}, \label{equ:Poisson}
\end{eqnarray} 
where  $\xi_i \in H^2(\Omega_i)$.  
Application of integration by parts twice in Eq.~\eqref{equ:Poisson} yields
\begin{eqnarray}
    \| \nabla\left(\nabla  \xi_i\right) \|_{L^2(\Omega_i)} ^2 =  \|\Delta  \xi_i \|_{L^2(\Omega_i)}^2 .\label{equ:H2tolaplace}
\end{eqnarray}
We define ${\vec{y}}_i: = \nabla \xi_i$ for $i\in \{\FF,\TR\}$.  Considering \eqref{equ:Poisson}, \eqref{equ:H2tolaplace} and \eqref{equ:auxiliarypoincare}, we obtain 
\begin{eqnarray*}
    \| \vec y _i\|_{H_i}^2 =\| {\vec y }_i\|_{L^2(\Omega_i)}^2 + \| \nabla {\vec y }_i \|_{L^2(\Omega_i)}^2 \le  \tilde{C}_{P,i} \|\nabla {\vec y }_i \|_{L^2(\Omega_i)}^2
     = \tilde{C}_{P,i} \| \nabla(\nabla  \xi_i) \|_{L^2(\Omega_i)}^2   = \tilde{C}_{P,i} \|\Delta  \xi_i \|_{L^2(\Omega_i)}^2 = \tilde{C}_{P,i} \| \psi_i \|_{L^2(\Omega_i)}^2.
\end{eqnarray*}
This leads to
\begin{eqnarray}
\| ({\vec y}_\FF, {\vec y}_\TR, 0)\|_{\mathcal{H}^\FU}^2 = \| {\vec y }_\FF\|_{H_\FF}^2+\| {\vec y }_\TR\|_{H_\TR}^2\le \left(  \tilde{C}_{P,\FF}+\tilde{C}_{P,\TR}\right)\| (\psi_\FF, \psi_\TR )\|_{\mathcal{Z}^\FU}^2.  \label{equ:infsup2-1}
\end{eqnarray}
Considering $\nabla \cdot \vec{y}_i = \Delta \xi_i = - \psi_i$, we get
\begin{eqnarray}
    \mathcal{B}^\FU(({\vec{y}}_\FF, {\vec{y}}_\TR, 0); \psi_\FF, \psi_\TR)=-\int_{\Omega_\FF} \left(\nabla \cdot {\vec{y}}_\FF\right) \psi_\FF~\mathrm{d}\vec{x}  -\int_{\Omega_\TR} \left(\nabla \cdot {\vec{y}}_\TR\right) \psi_\TR~\mathrm{d}\vec{x}    =\| \psi_\FF \|_{L^2(\Omega_\FF)}^2 +  \| \psi_\TR \|_{L^2(\Omega_\TR)}^2 =\| (\psi_\FF, \psi_\TR)\|^2_{\mathcal{Z}^\FU}.\label{equ:infsup2-2}
\end{eqnarray}
Taking equations \eqref{equ:infsup2-1} and \eqref{equ:infsup2-2} into account, we obtain 
\begin{eqnarray*}
    \sup_{\vec{\zeta}^\FU\in \mathcal{H}^\FU}  \frac{\mathcal{B}^\FU(\vec{\zeta}^\FU; \psi_\FF, \psi_\TR) }{\| \vec{\zeta}^\FU\|_{\mathcal{H}^\FU}} \ge  \frac{\mathcal{B}^\FU( ({\vec{y}}_\FF, {\vec{y}}_\TR, 0); \psi_\FF, \psi_\TR)}{ \|( {\vec{y}}_\FF, {\vec{y}}_\TR, 0)\|_{\mathcal{H}^\FU}}   \ge  C^*_{\mathcal{B}^\FU} \|(\psi_\FF, \psi_\TR) \|_{\mathcal{Z}^\FU},\qquad
 C_{\mathcal{B}^\FU}^*:= \left(\tilde{C}_{P,\FF}+\tilde{C}_{P,\TR}\right)^{-1/2}. 
\end{eqnarray*}
Thus, $\mathcal{B}^\FU$ is inf-sup-stable. This completes the proof. 
\end{proof}


\subsection{Weak formulation of the dimensionally reduced model}
\label{sec:weak-reduced}
In this section, we derive the weak formulation of the coupled dimensionally reduced model. The boundary $\partial \gamma =  \partial \gamma_D \cup \partial \gamma_N$ (Fig.~\ref{fig:Fig1intro}, right) consists of two points, where either the averaged Dirichlet or the averaged Neumann boundary conditions~\eqref{equ:TRBC} are prescribed. 
We consider test function spaces
$\displaystyle \vec{W}\in H_\gamma : = \left\{ \vec{W}\in \left( H^1\left(\gamma\right)\right)^2: \vec{W}|_{\partial \gamma_D} = \vec{0}\right\}$ and $\displaystyle \Psi \in Z_\gamma := L^2 \left(\gamma\right)$ 
equipped with the norms 
$\displaystyle \|\vec{W}\|_{H_\gamma}^2:= \|\vec{W} \|_{L^2(\gamma)}^2+\left\|(\partial \vec{W}) \big/ (\partial \vec{\tau}) \right\|_{L^2(\gamma)}^2$ and $\displaystyle \| \Psi \|_{Z_\gamma} :=\| \Psi \|_{L^2(\gamma)}$.

\subsubsection{Flow problems in \texorpdfstring{$\Omega_\FF$}{Omega_ff} and \texorpdfstring{$\Omega_\PM$}{Omega_pm}}
We begin with the weak form for the flow problems in the free-flow and porous-medium subdomains. 
Decomposing the normal stress over $\gamma_\FF$ in~Eq.~\eqref{equ:weakformulationofStokesMomentum} and  applying the transmission conditions \eqref{equ:transmissionFFnormal} and  \eqref{equ:transmissionPM}, we get  
\begin{eqnarray}
\int_{\gamma_{\FF}}\left( \vec{n}\cdot \ten{T}\left(\vec{v}_\FF, p_\FF\right) \cdot \vec{n}\right)  (\vec{w}_\FF\cdot \vec{n})~\mathrm{d} s 
      =\int_{\gamma_\FF} \Bigg( -\frac{ \mu_\EF\left(\lambda_1^2-\lambda_2^2\right)}{\lambda_1 d} V_\vec{n}- \frac{\lambda_1 +\lambda_2 }{\lambda_1} P +   \frac{\mu_\EF(\lambda_1^2 -\lambda_2^2)}{\lambda_1 d} \vec{v}_\FF\cdot \vec{n} + \frac{\lambda_2}{\lambda_1}p_\PM |_{\gamma_\PM}\nonumber\hspace{+6.5ex}\\
       + \frac{\mu }{\sqrt{K_{\TR}}}\left(\vec{\beta}\vec{v}_\FF\right)\cdot\vec{n}  \Bigg)\left(\vec{w}_\FF \cdot \vec{n}\right)~\mathrm{d} s .\label{equ:reducedweakformulationFFTensornormal}\hspace{+38ex}
\end{eqnarray}
Analogously, substituting the transmission condition~\eqref{equ:transmissionFFtangential} into the tangential part of the normal stress tensor over $\gamma_\FF$ in \eqref{equ:weakformulationofStokesMomentum}, we obtain
\begin{eqnarray}
    \int_{\gamma_{\FF}}\left( \vec{\tau}\cdot \ten{T}\left(\vec{v}_\FF, p_\FF\right) \cdot \vec{n}\right)  \left(\vec{w}_\FF\cdot \vec{\tau}\right)~\mathrm{d} s
    =\int_{\gamma_\FF} \Bigg( - \frac{ \mu_\EF\left(6 \alpha d + 12\sqrt{K_\PM}\right) }{d\left(\alpha d+4 \sqrt{K_\PM}\right)} V_\vec{\tau}
     +\frac{  \mu_\EF\left(4 \alpha d + 12 \sqrt{{K}_\PM}\right) }{d\left( \alpha d+4\sqrt{{K}_\PM} \right)}\vec{v}_\FF \cdot \vec{\tau}\hspace{+11ex}
     \nonumber\\
     + \frac{\mu }{\sqrt{{K}}_\TR} \left(\vec{\beta}\vec{v}_\FF\right) \cdot \vec{\tau}  \Bigg)\left(\vec{w}_\FF \cdot \vec{\tau}\right)~\mathrm{d} s. \hspace{+35.5ex} \label{equ:reducedweakformulationFFTensortangential}
\end{eqnarray}
%
Using the transmission condition \eqref{equ:transmissionPM} in the normal component of the porous-medium velocity over $\gamma_\PM$ in Eq.~\eqref{equ:weakformulationofDarcy}, we get
\begin{eqnarray}
     \int_{\gamma_\PM} \left(\vec{v}_\PM \cdot \vec{n}\right)\varphi_\PM ~\mathrm{d} s=\int_{\gamma_\PM} \Bigg( \frac{\lambda_1 + \lambda_2}{\lambda_1}V_\vec{n} - \frac{d}{\mu_\EF\lambda_1} P 
     +\frac{d}{\mu_\EF\lambda_1}p_\PM- \frac{\lambda_2}{\lambda_1} \vec{v}_\FF\cdot \vec{n} |_{\gamma_\FF} \Bigg) \varphi_\PM ~\mathrm{d} s. \label{equ:reducedweakformulationPMTtransmission}
\end{eqnarray}

Substituting \eqref{equ:reducedweakformulationFFTensornormal}--\eqref{equ:reducedweakformulationPMTtransmission} in \eqref{equ:weakformulationofStokesMomentum}--\eqref{equ:weakformulationofDarcy}, we obtain the weak form for the free-flow and porous-medium subdomains:\\
\emph{Find $(\vec{v}_\FF, p_\PM) \in H_\FF \times H_\PM $ and $p_\FF\in Z_\FF$ such that }
\begin{subequations}
\begin{eqnarray}
    \hspace{-2ex}\mathcal{A}_{\FP} (\vec{v}_\FF, p_\PM; \vec{w}_\FF, \varphi_\PM) + \mathcal{B}_{\FP}(\vec{w}_\FF, \varphi_\PM;p_\FF )&=&f_{\FP} (\vec{V}, P; \vec{w}_\FF, \varphi_\PM)+\mathcal{L}_{\FP}(\vec{w}_\FF, \varphi_\PM), \hspace{+1ex} \forall (\vec{w}_\FF, \varphi_\PM)\in H_\FF\times H_\PM,\label{equ:WeakUncoupledFlow1}\\
    \mathcal{B}_{\FP}( \vec{v}_\FF, p_\PM;\psi_\FF )&= &0 , \hspace{+30ex}\forall \psi_\FF\in Z_\FF.
 \label{equ:WeakUncoupledFlow2}
\end{eqnarray} \label{equ:WeakUncoupledFlow}
\end{subequations}
 Here, the bilinear forms
\begin{eqnarray*}
    \mathcal{A}_{\FP} (\vec{v}_\FF, p_\PM; \vec{w}_\FF, \varphi_\PM) &:=&\mathcal{A}_\FF ( \vec{v}_\FF;  \vec{w}_\FF) +\mathcal{A}_\PM (p_\PM; \varphi_\PM)
    + \mathcal{A}_{\FP,\gamma}(\vec{v}_\FF, p_\PM; \vec{w}_\FF, \varphi_\PM) \\
    \mathcal{B}_{\FP}( \vec{v}_\FF, p_\PM;\psi_\FF ) &:=& - \int_{\Omega_\FF} \left(\nabla \cdot \vec{v}_\FF \right)\psi_\FF~\mathrm{d}\vec{x},
    \end{eqnarray*}
are defined with
\begin{eqnarray*}
   \mathcal{A}_{\FP,\gamma}(\vec{v}_\FF, p_\PM; \vec{w}_\FF, \varphi_\PM):=   \int_{\gamma_\FF} \frac{\mu_\EF(\lambda_1^2 -\lambda_2^2)}{\lambda_1 d} \left(\vec{v}_\FF\cdot \vec{n} \right)\left(\vec{w}_\FF \cdot \vec{n}\right)~\mathrm{d} s +\int_{\gamma_\FF}  \frac{  \mu_\EF\left(4 \alpha d + 12\sqrt{{K}_\PM}\right) }{d\left( \alpha d+4\sqrt{{K}_\PM} \right)}\left( \vec{v}_\FF \cdot \vec{\tau} \right)\left(\vec{w}_\FF \cdot \vec{\tau}\right)~\mathrm{d} s \hspace{+1ex}\\
     + \int_{\gamma_\FF}\frac{\mu }{\sqrt{{K}_{\TR}}}(\vec{\beta}\vec{v}_\FF) \cdot\vec{w}_\FF~\mathrm{d} s +   \int_{\gamma_\FF}   \frac{\lambda_2}{\lambda_1}p_\PM |_{\gamma_\PM}  (\vec{w}_\FF \cdot \vec{n})~\mathrm{d} s
  -  \int_{\gamma_\PM}  \frac{\lambda_2}{\lambda_1} \left(\vec{v}_\FF\cdot \vec{n} |_{\gamma_\FF} \right) \varphi_\PM ~\mathrm{d} s +  \int_{\gamma_\PM} \frac{d}{\mu_\EF\lambda_1}p_\PM  \varphi_\PM ~\mathrm{d} s.
\end{eqnarray*}
Additionally, we have the bilinear operator
\begin{eqnarray*}
    f_{\FP} (\vec{V}, P; \vec{w}_\FF, \varphi_\PM) :=\int_{\gamma_\FF} \frac{\mu_\EF \left(\lambda_1^2-\lambda_2^2\right)}{\lambda_1 d} V_\vec{n} \left(\vec{w}_\FF \cdot \vec{n}\right)~\mathrm{d} s+\int_{\gamma_\FF} \frac{\lambda_1+\lambda_2 }{\lambda_1} P  \left(\vec{w}_\FF \cdot \vec{n}\right)~\mathrm{d} s\hspace{+27ex}\\
     +\int_{\gamma_\FF} \frac{ \mu_\EF\left(6 \alpha d + 12\sqrt{K_\PM}\right) }{d\left( \alpha d+4\sqrt{{K}_\PM}\right)} V_\vec{\tau}\left(\vec{w}_\FF \cdot \vec{\tau}\right)~\mathrm{d} s+\int_{\gamma_\PM}\frac{d}{\mu_\EF\lambda_1} P \varphi_\PM ~\mathrm{d} s-\int_{\gamma_\PM}\frac{\lambda_1 + \lambda_2}{\lambda_1}V_\vec{n} \varphi_\PM ~\mathrm{d} s,
\end{eqnarray*}
and the linear functional
    \begin{eqnarray*}
    \mathcal{L}_{\FP}(\vec{w}_\FF, \varphi_\PM) :=\int_{\Omega_\FF} \vec{f}_\FF \cdot \vec{w}_\FF ~\mathrm{d}\vec{x}
    +\int_{\Omega_\PM}q \varphi_\PM ~\mathrm{d}\vec{x}+\int_{\Gamma_{N,\FF}}\overline{\vec{t}}_\FF \cdot \vec{w}_\FF~\mathrm{d} s-\int_{\Gamma_{N,\PM}}\overline{v}_\PM \varphi_\PM ~\mathrm{d} s.
\end{eqnarray*}
According to the assumption on the geometry (section~\ref{sec:derivation-dimred}) the interfaces $\gamma_\FF$ and $\gamma_\PM$ approach the complex interface~$\gamma$ when $d\to 0$. Therefore, the integrals over $\gamma_\FF$ and $\gamma_\PM$ in  $\mathcal{A}_{\FP,\gamma}$, $f_\FP$ from \eqref{equ:WeakUncoupledFlow1} are considered as the integrals over $\gamma$. 

\subsubsection{Reduced problem on \texorpdfstring{$\gamma$}{gamma}}
To derive the weak form for the reduced problem on $\gamma$, we start with the boundary conditions on~$\partial \gamma= \partial \gamma_D \cup \partial \gamma_N$. 
For the Dirichlet boundary conditions, we define 
\begin{eqnarray}
     \overline{\vec V} :=\frac{1}{d} \int^{d/2}_{-d/2} \overline{\vec{v}}_\TR~\mathrm{d} n \quad \textnormal{on } \partial \gamma_D, \label{equ:averageDirichlet} 
\end{eqnarray}
where $\overline{\vec{V}}\in \left(H^{1/2}\left(\partial \gamma_D\right)\right)^2$. Averaging the Neumann boundary condition in~\eqref{equ:TRBC} across the transition region, we obtain
\begin{eqnarray}
    \overline{T}_{\vec{n}} = \frac{1}{d} \int^{d/2}_{-d/2} \overline{\vec{t}}_\TR \cdot \vec{n}~\mathrm{d} n = \mu_\EF \frac{\partial V_\vec{n}}{\partial \vec \tau},\quad 
    \overline{T}_\vec{\tau} = \frac{1}{d} \int^{d/2}_{-d/2}\overline{\vec{t}}_\TR \cdot \vec{\tau} ~\mathrm{d} n =\mu_\EF \frac{\partial V_\vec{\tau}}{\partial \vec \tau}-P \quad \textnormal{on } \partial \gamma_N, \label{equ:averageNeuman} 
\end{eqnarray}
where
$\overline{\vec{T}}=\left(\overline{T}_\vec{n}, \overline{T}_\vec{\tau}\right)^\top \in \left(H^{-1/2}(\partial \gamma_N)\right)^2$.

We derive the weak formulation for the averaged mass conservation equation~\eqref{equ:averagedMass} in the standard way, i.e. multiply it by the test function $\Psi \in Z_\gamma$, integrate over $\gamma$ and consider the transmission condition~\eqref{equ:transmissionPM}. The result reads
\begin{eqnarray}
    \int_\gamma d \frac{\partial V_\vec{\tau}}{\partial \vec{\tau}}  \Psi ~\mathrm{d} s   
    =- \int_\gamma \frac{\lambda_1+\lambda_2}{\lambda_1}\left(\vec{v}_\FF\cdot \vec{n} | _{\gamma_\FF}\right)\Psi~\mathrm{d} s +\int_\gamma  \frac{d}{\mu_\EF\lambda_1}  \left(p_\PM |_{\gamma_\PM} \right) \Psi~\mathrm{d} s  -\int_\gamma  \frac{d}{\mu_\EF\lambda_1} P\Psi~\mathrm{d} s  +\int_\gamma  \frac{\lambda_1 + \lambda_2}{\lambda_1}V_\vec{n}  \Psi~\mathrm{d} s.\label{equ:weakformlationAveragedBrinkmanMass}
\end{eqnarray}
Substituting Eqs.~\eqref{equ:constantpressure}, \eqref{equ:generalnormalclosure-1} and \eqref{equ:generalnormalclosure-2} into Eq.~\eqref{equ:averagedMomentumnormaloriginal}, we get the following equivalent formulation for the normal component of the averaged momentum conservation equation 
\begin{eqnarray}
    \left(\frac{2\mu_\EF(\lambda_1 +\lambda_2)}{d} + d\mu \vec{M}_{\vec{n}\vec{n}}\right) V_{\vec{n}}+ d\mu \vec{M}_{\vec{n}\vec{\tau}} V_\vec{\tau} - d\mu_\EF \frac{\partial^2 V_\vec{n}}{\partial \vec{\tau}^2} = d F_\vec{n} +\frac{\mu_\EF\left(\lambda_1+\lambda_2\right)}{d} \vec{v}_\FF \cdot \vec{n}|_{\gamma_\FF}+ \frac{\mu_\EF( \lambda_1 + \lambda_2)}{d} \vec{v}_\PM\cdot \vec{n}|_{\gamma_\PM}.\label{equ:aveBrinkmannormalUpdate}
\end{eqnarray}
We consider test functions $ \vec{W} := \left(W_{\vec{n}},W_{\vec{\tau}}\right)^\top \in H_\gamma$. Multiplying Eq.~\eqref{equ:aveBrinkmannormalUpdate} with  $W_\vec{n}$, integrating over $\gamma$ by parts and substituting the normal component of the porous-medium velocity from Eq.~\eqref{equ:transmissionPM}, we obtain
\begin{eqnarray}
    \int_\gamma\left(\frac{\mu_\EF\left(\lambda_1^2 -\lambda_2^2\right)}{\lambda_1 d} + d\mu \vec{M}_{\vec{n}\vec{n}}\right) V_{\vec{n}}W_{\vec{n}}~\mathrm{d} s+  \int_\gamma d\mu \vec{M}_{\vec{n}\vec{\tau}} V_\vec{\tau}W_{\vec{n}}~\mathrm{d} s 
      +  \int_\gamma d\mu_\EF \frac{\partial V_\vec{n}}{\partial \vec{\tau}} \frac{\partial W_{\vec{n}}}{\partial \vec\tau }~\mathrm{d} s +\int_\gamma \frac{ \lambda_1 + \lambda_2}{\lambda_1 } P W_{\vec{n}}~\mathrm{d} s  \label{equ:weakformlationAveragedBrinkmanMomentumNormal}\hspace{+6ex}\\
      =  \int_\gamma d F_\vec{n}W_{\vec{n}}~\mathrm{d} s + d\left[ \overline{T}_{\vec n} W_{\vec n}\right]_{\partial \gamma_N}
      +  \int_\gamma\frac{\mu_\EF\left(\lambda_1^2-\lambda_2^2\right)}{\lambda_1 d} \left(\vec{v}_\FF \cdot \vec{n}|_{\gamma_\FF}\right) W_{\vec{n}}~\mathrm{d} s
     +\int_\gamma\frac{ \lambda_1 + \lambda_2}{\lambda_1 }\left( p_\PM|_{\gamma_\PM} \right) W_{\vec{n}}~\mathrm{d} s. \nonumber
\end{eqnarray}
Analogously, we obtain the tangential component of the averaged momentum conservation equation on $\gamma$ by substituting \eqref{equ:BJS-ICPM}, \eqref{equ:closure_tangential1} and \eqref{equ:closure_tangential2} into~\eqref{equ:averagedMomentumtangentialoriginal}:
\begin{eqnarray}
    d \mu \vec{M}_{\vec{\tau}\vec{n}} V_{\vec{n}}+ \left(\frac{\mu_\EF\left(12 \alpha d + 12\sqrt{K_\PM}\right)}{d \left(\alpha d+4\sqrt{K_\PM}\right)}  +  d\mu \vec{M}_{\vec{\tau\tau}} \right) V_\vec{\tau} 
    -  d \left(\mu_\EF   \frac{\partial^2 V_\vec{\tau}}{\partial \vec{\tau}^2} - \frac{\partial P}{\partial \vec{\tau}}\right) 
    =  d  F_\vec{\tau}+ \frac{\mu_\EF\left(6\alpha d +12\sqrt{K_\PM}\right)}{d \left(\alpha d+4\sqrt{K_\PM}\right)} \vec{v}_\FF \cdot \vec{\tau} |_{\gamma_\FF}. \label{equ:aveBrinkmantangentialUpdate}
    \end{eqnarray} 
Multiplying Eq.~\eqref{equ:aveBrinkmantangentialUpdate} with the test function $W_\vec{\tau}$, integrating over $\gamma$ by parts, we have
\begin{eqnarray}
     \int_\gamma d \mu \vec{M}_{\vec{\tau}\vec{n}} V_{\vec{n}}W_{\vec{\tau}}~\mathrm{d} s+ \int_\gamma\left(\frac{\mu_\EF\left(12 \alpha  d + 12\sqrt{K_\PM}\right)}{d \left(\alpha d+4\sqrt{K_\PM}\right)}  +  d\mu \vec{M}_{\vec{\tau\tau}} \right) V_\vec{\tau} W_{\vec{\tau}}~\mathrm{d} s + \int_\gamma   d\left(  \mu_\EF \frac{\partial V_\vec{\tau}}{\partial \vec{\tau}} -   P\right)\frac{\partial W_{\vec{\tau}}}{\partial \vec{\tau}}~\mathrm{d} s  \label{equ:weakformlationAveragedBrinkmanMomentumtangential} \hspace{+10ex} \\
     =\int_\gamma  d  F_\vec{\tau} W_{\vec{\tau}}~\mathrm{d} s+d\left[\overline{T}_\vec{\tau}W_{\vec \tau} \right]_{\partial \gamma_N}  
     + \int_\gamma \frac{\mu_\EF\left(6\alpha d +12\sqrt{K_\PM}\right)}{ d \left(\alpha d+4\sqrt{K_\PM}\right)} \left(\vec{v}_\FF \cdot \vec{\tau} |_{\gamma_\FF}\right) W_{\vec{\tau}}~\mathrm{d} s. \nonumber 
\end{eqnarray}

The uncoupled weak formulations on $\gamma$ given in Eqs.~\eqref{equ:weakformlationAveragedBrinkmanMass}, \eqref{equ:weakformlationAveragedBrinkmanMomentumNormal} and \eqref{equ:weakformlationAveragedBrinkmanMomentumtangential} can be written as:\\
\emph{Find $\vec{V}=(V_{\vec n},V_\vec{\tau})^\top \in H_\gamma$, $\Psi \in Z_\gamma$ such that}
\begin{subequations}
\begin{eqnarray}
        \mathcal{A}_\gamma\left( \vec{V}; \vec{W}\right)   + \mathcal{B}_\gamma\left(\vec{W}; P\right)  &=&  f_\gamma( \vec{v}_\FF, p_\PM;\vec{W})+\mathcal{L}_\gamma\left( \vec{W}\right) , \quad  \forall\vec{W} \in H_\gamma,  \label{equ:weakUncoupledgamma1} \\
    \mathcal{B}_\gamma \left( \vec{V}; \Psi \right) - \mathcal{E}_\gamma\left(P;\Psi \right)&=& g_\gamma (\vec{v}_\FF, p_\PM; \Psi ), \hspace{+12.2ex} \forall \Psi\hspace{+0.5ex}\in\hspace{+0.2ex} Z_\gamma. \label{equ:weakUncoupledgamma2}
\end{eqnarray}\label{equ:weakUncoupledgamma}
\end{subequations}
Here, the following bilinear operators and the linear functionals are given  
\begin{eqnarray*}
    \mathcal{A}_\gamma\left( \vec{V}; \vec{W}\right) &:=&    \int_\gamma d\mu (\ten{K}_\TR^{-1} \vec{V}) \cdot \vec{W}~\mathrm{d} s +\int_\gamma \frac{\mu_\EF\left(\lambda_1^2 -\lambda_2^2\right)}{\lambda_1 d} V_{\vec{n}}W_{\vec{n}}~\mathrm{d} s +  \int_\gamma d\mu_\EF \frac{\partial V_\vec{n}}{\partial \vec{\tau}} \frac{\partial W_{\vec{n}}}{\partial \vec\tau }~\mathrm{d} s  \\
     &\quad &+ \int_\gamma   d  \mu_\EF \frac{\partial V_\vec{\tau}}{\partial \vec{\tau}}\frac{\partial W_{\vec{\tau}}}{\partial \vec{\tau}}~\mathrm{d} s+ \int_\gamma \frac{\mu_\EF\left(12 \alpha  d + 12\sqrt{K_\PM}\right)}{d \left( \alpha d+4\sqrt{K_\PM}\right)} V_\vec{\tau} W_{\vec{\tau}}~\mathrm{d} s,  \\
    \mathcal{B}_\gamma\left(\vec{W}; P\right)  &:=& - \int_\gamma   d P \frac{\partial W_{\vec{\tau}}}{\partial \vec{\tau}}~\mathrm{d} s +\int_\gamma   \frac{ \lambda_1 + \lambda_2}{\lambda_1 } P  W_{\vec{n}}~\mathrm{d} s, \qquad
    \mathcal{E}_\gamma\left(P;\Psi \right) := \int_\gamma \frac{d}{\mu_\EF\lambda_1} P  \Psi ~\mathrm{d} s,\\
     f_\gamma( \vec{v}_\FF, p_\PM;\vec{W})&:=&  \int_\gamma  \frac{\mu_\EF\left(\lambda_1^2-\lambda_2^2\right)}{\lambda_1 d}(\vec{v}_\FF \cdot \vec{n}|_{\gamma_\FF})W_{\vec{n}}~\mathrm{d} s  +  \int_\gamma  \frac{( \lambda_1 + \lambda_2)}{\lambda_1 }\left(p_\PM|_{\gamma_\PM}   \right) W_{\vec{n}}~\mathrm{d} s \\
     &\quad& + \int_\gamma \frac{\mu_\EF\left(6 \alpha d +12\sqrt{K_\PM}\right)}{d \left(\alpha d+4\sqrt{K_\PM}\right)} \left(\vec{v}_\FF \cdot \vec{\tau} |_{\gamma_\FF}\right) W_{\vec{\tau}}~\mathrm{d} s, \\
    g_\gamma (\vec{v}_\FF, p_\PM;\Psi )&:=& \int_\gamma \frac{\lambda_1 +\lambda_2}{\lambda_1} \left(\vec{v}_\FF\cdot \vec{n} | _{\gamma_\FF}\right)\Psi~\mathrm{d} s -\int_\gamma \frac{d}{\mu_\EF\lambda_1}\left(p_\PM |_{\gamma_\PM}\right)\Psi~\mathrm{d} s,\\
     \mathcal{L}_\gamma\left( \vec{W}\right)&:=& \int_\gamma d F_\vec{n}W_{\vec{n}}~\mathrm{d} s +\int_\gamma  d  F_\vec{\tau} W_{\vec{\tau}}~\mathrm{d} s+ d\left[ \overline{\vec{T}} \cdot \vec{W} \right]_{\partial \gamma_N}. 
\end{eqnarray*}

\subsection{Analysis of the dimensionally reduced model}
\label{sec:ana-reduced}
For the well-posedness analysis of the dimensionally reduced model, we define the test function spaces
$\displaystyle \vec \theta ^\RE: =( \vec{w}_\FF, \varphi_\PM, \vec{W}, \Psi) \in \mathcal{H}^\RE := H_\FF \times H_\PM \times H_\gamma\times Z_\gamma $ and  $\displaystyle \psi_\FF \in \mathcal{Z}^\RE:= Z_\FF$
equipped with the norms $\displaystyle  \| \vec \theta ^\RE \|_{\mathcal{H}^\RE}^2 :=  \| \vec{w}_\FF \|_{H_\FF}^2  + \|\varphi_\PM \|_{H_\PM}^2+ \| \vec{W}\|_{H_\gamma}^2 + \|\Psi\|_{Z_\gamma}^2$ and $\displaystyle  \| \psi_\FF \|_{\mathcal{Z}^\RE} :=  \|\psi_\FF\|_{Z_\FF}$, respectively.
Taking into account Eqs.~\eqref{equ:WeakUncoupledFlow} and \eqref{equ:weakUncoupledgamma}, we obtain the weak formulation of the coupled dimensionally reduced model:\\
\emph{Find $\vec{\zeta}^\RE:=(\vec{v}_\FF, p_\PM, \vec{V},P)\in \mathcal{H}^\RE $ and  $p_\FF \in Z^\RE$ such that}
\begin{subequations}
\begin{eqnarray}
    \mathcal{A}^\RE \left( \vec{\zeta}^\RE ; \vec{\theta}^\RE \right) + \mathcal{B}^\RE \left(\vec{\theta}^\RE; p_\FF \right) &=& \mathcal{L}^\RE (\vec \theta^\RE), 
    \quad  \forall \vec{\theta}^\RE \in \mathcal{H}^\RE, \label{equ:weakformulationofCoupledreducedproblem1}\\
     \mathcal{B}^\RE \left(\vec{\zeta}^\RE; \psi_\FF \right) &=& 0,
     \hspace{+7.6ex}
     \forall \psi_\FF \in \mathcal{Z}^\RE, \label{equ:weakformulationofCoupledreducedproblem2}
\end{eqnarray}%
\label{eq:weakformulationofCoupledreducedproblem}\end{subequations}
where the bilinear operators $\mathcal{A}^\RE: \mathcal{H}^\RE \times \mathcal{H}^\RE \to \mathbb{R}$, $\mathcal{B}^\RE: \mathcal{H}^\RE \times \mathcal{Z}^\RE \to \mathbb{R} $ and the linear functional $\mathcal{L}^\RE: \mathcal{H}^\RE \to \mathbb{R}$ are
\begin{eqnarray*}
    \mathcal{A}^\RE \left( \vec{\zeta}^\RE ; \vec{\theta}^\RE \right) &:=&
    \mathcal{A}_{\FP} (\vec{v}_\FF, p_\PM; \vec{w}_\FF, \varphi_\PM) 
    +\mathcal{A}_\gamma\left( \vec{V}; \vec{W}\right)   + \mathcal{B}_\gamma\left(\vec{W}; P\right) - \mathcal{B}_\gamma \left( \vec{V}; \Psi \right) + \mathcal{E}_\gamma \left(P; \Psi \right)\\
    &\quad&- f_{\FP} (\vec{V}, P; \vec{w}_\FF, \varphi_\PM) - f_\gamma( \vec{v}_\FF, p_\PM;\vec{W}) +g_\gamma (\vec{v}_\FF, p_\PM;\Psi) ,\\
    \mathcal{B}^\RE \left(\vec{\zeta}^\RE; \psi_\FF \right) &:=& \mathcal{B}_\FP (\vec{v}_\FF, p_\PM; \psi_\FF),
   \qquad \mathcal{L}^\RE(\vec \theta^\RE): =  \mathcal{L}_{\FP} (\vec{w}_\FF, \varphi_\PM)+\mathcal{L}_\gamma(\vec{W}) .
\end{eqnarray*}
\emph{Remark 2.} We treat the non-homogeneous Dirichlet boundary conditions $(\overline{\vec{v}}_\FF, \overline{p}_\PM, \overline{\vec V}) \in \left(H^{1/2}\left(\Gamma_{D,\FF}\right)\right)^2 \times H^{1/2}(\Gamma_{D,\PM}) \times \left(H^{1/2}\left(\partial\gamma_{D}\right)\right)^2$ by considering the corresponding lifting $\Xi^\RE \in \left(H^1(\Omega_\FF)\right)^2 \times H^1(\Omega_\PM) \times \left(H^1(\gamma)\right)^2$ of $(\hat{\vec{v}}_\FF, \hat{p}_\PM, \hat{\vec{V}})\in \left(H^1(\Omega_\FF)\right)^2 \times H^1(\Omega_\PM) \times \left(H^1(\gamma)\right)^2$ such that $((\hat{\vec{v}}_\FF, \hat{p}_\PM, \hat{\vec{V}})-\Xi^\RE,P)=\vec{\zeta}^\RE\in \mathcal{H}^\RE$ in the formulation~\eqref{eq:weakformulationofCoupledreducedproblem}.

\begin{theorem}[Well-posedness of dimensionally reduced model]\label{thm:reduceddimwellposedness}
There exists a positive parameter $d^*\in \mathbb{R}^+$. Under the assumption the transition region thickness satisfies \mbox{$d \in (0, d^*)$}, the solution 
    of the coupled dimensionally reduced problem \eqref{eq:weakformulationofCoupledreducedproblem} exists and is unique. 
\end{theorem} 
\begin{proof}
To show the well-posedness of \eqref{eq:weakformulationofCoupledreducedproblem}, we verify conditions \eqref{continuous1}--\eqref{inf-sup} for $\mathcal{A}^\RE$ and $\mathcal{B}^\RE$. Following the same steps as in the proof of Theorem~\ref{thm:fulldimwellposedness} for $\mathcal{B}^\FU$, we show $\mathcal{B}^\RE$ is continuous and inf-sup stable.

To verify the continuity of $\mathcal{A}^\RE$, we first need to find the upper bounds of $\mathcal{A}_\FP$ in Eq.~\eqref{equ:WeakUncoupledFlow1} and $\mathcal{A}_\gamma$, $\mathcal{B}_\gamma$, $\mathcal{E}_\gamma$ in Eqs.~\eqref{equ:weakUncoupledgamma}.  Taking into account the Cauchy--Schwarz inequality, the uniform ellipticity conditions \eqref{equ:boundedK}, \eqref{equ:boundedbeta}, the trace theorem results~\eqref{equ:traceFFTR}, \eqref{equ:tracePM} and the fact that $\lambda_1 > \lambda_2 \ge 0$ (see Table~\ref{tab:closureconditionnormal}), we obtain 
\begin{eqnarray}
 |\mathcal{A}_{\FP} (\vec{v}_\FF, p_\PM; \vec{w}_\FF, \varphi_\PM)| 
 \le | \mathcal{A}_\FF \left( \vec{v}_\FF;  \vec{w}_\FF\right)| +|\mathcal{A}_\PM (p_\PM; \varphi_\PM)|  + |\mathcal{A}_{\FP,\gamma} (\vec{v}_\FF, p_\PM; \vec{w}_\FF, \varphi_\PM)|  \hspace{+22ex} \nonumber\\
     \le\mu \| \nabla \vec{v}_\FF \|_{L^2(\Omega_\FF)} \| \nabla\vec{w}_\FF \|_{L^2 (\Omega_\FF)} + \frac{k_{\max, \PM}}{\mu} \|\nabla p_\PM\|_{L^2(\Omega_\PM)}\|\nabla \varphi_\PM\|_{L^2(\Omega_\PM)}+ \frac{\mu_\EF\left(\lambda_1^2 -\lambda_2^2\right)}{\lambda_1 d} \| \vec{v}_\FF\cdot \vec{n}\|_{L^2(\gamma)} \| \vec{w}_\FF \cdot \vec{n}\|_{L^2(\gamma)}\nonumber\\ 
    +  \frac{  \mu_\EF\left(4 \alpha d + 12\sqrt{{K}_\PM}\right) }{d\left( \alpha d+4\sqrt{{K}_\PM} \right)}\| \vec{v}_\FF \cdot \vec{\tau} \|_{L^2(\gamma)} \|\vec{w}_\FF \cdot \vec{\tau}\|_{L^2(\gamma)}
     + \frac{\mu \beta_{\max}}{\sqrt{{K}_{\TR}}}\| \vec{v}_\FF\|_{L^2(\gamma)} \| \vec{w}_\FF\|_{L^2(\gamma)}  +    \frac{\lambda_2}{\lambda_1} \| p_\PM \|_{L^2(\gamma)}  \|\vec{w}_\FF \cdot \vec{n}\|_{L^2(\gamma)} \nonumber\\
   +    \frac{\lambda_2}{\lambda_1} \| \vec{v}_\FF\cdot \vec{n} \|_{L^2(\gamma)}\|\varphi_\PM\|_{L^2(\gamma)}   +   \frac{d}{\mu_\EF\lambda_1} \| p_\PM\|_{L^2(\gamma)}  \|\varphi_\PM \|_{L^2(\gamma)}\hspace{+41.5ex}\nonumber
   \\
   %
   \le \left(\mu  + \left( \frac{\mu_\EF\left(\lambda_1^2 -\lambda_2^2\right)}{\lambda_1 d} + \frac{  \mu_\EF\left(4 \alpha d + 12\sqrt{{K}_\PM}\right) }{d\left( \alpha d+4\sqrt{{K}_\PM} \right)}+ \frac{\mu \beta_{\max}}{\sqrt{{K}_{\TR}}}\right)  C_{\gamma_\FF,\FF}^2\right) \|  \vec{v}_\FF \|_{H_\FF} \| \vec{w}_\FF \|_{H_\FF} \hspace{+23.5ex}\nonumber \\
    +    \frac{\lambda_2}{\lambda_1} C_{\gamma_\PM,\PM}  C_{\gamma_\FF,\FF}\| p_\PM \|_{H_\PM}  \|\vec{w}_\FF \|_{H_\FF} 
    +  \frac{\lambda_2}{\lambda_1} C_{\gamma_\PM,\PM}  C_{\gamma_\FF,\FF} \| \vec{v}_\FF \|_{H_\FF}\|\varphi_\PM\|_{H_\PM}  \hspace{+30.5ex}  \nonumber\\
    + \left(\frac{k_{\max, \PM}}{\mu} +  \frac{d}{\mu_\EF\lambda_1}  C_{\gamma_\PM,\PM}^2\right) \| p_\PM\|_{H_\FF}  \|\varphi_\PM \|_{H_\PM}\hspace{+49.8ex}\nonumber\\
    \le C_{\mathcal{A}_\FP}\left( \| \vec{v}_\FF\|_{H_\FF} + \| p_\PM\|_{H_\PM} \right)\left(\| \vec{w}_\FF \|_{H_\FF} +\| \varphi_\PM\|_{H_\PM}\right) \hspace{+49ex} \label{equ:limit_A_fp}
   %
\end{eqnarray}
with 
\begin{eqnarray*}
C_{\mathcal{A}_\FP} :=\max \left\{ \mu  + \left( \frac{\mu_\EF\left(\lambda_1^2 -\lambda_2^2\right)}{\lambda_1 d} + \frac{  \mu_\EF\left(4 \alpha d + 12\sqrt{{K}_\PM}\right) }{d\left( \alpha d+4\sqrt{{K}_\PM} \right)}+ \frac{\mu \beta_{\max}}{\sqrt{{K}_{\TR}}}\right)  C_{\gamma_\FF,\FF}^2, \,   \frac{\lambda_2}{\lambda_1} C_{\gamma_\PM,\PM}  C_{\gamma_\FF,\FF},  \,   \frac{k_{\max, \PM}}{\mu}+\frac{d}{\mu_\EF\lambda_1}  C_{\gamma_\PM,\PM}^2 \right\}.
\end{eqnarray*}
In a similar manner, we estimate
\begin{eqnarray}
        |\mathcal{A}_\gamma\left( \vec{V}; \vec{W}\right)| \le   \frac{ d\mu}{k_{\min,\TR}} \| \vec{V}\|_{L^2(\gamma)} \| \vec{W}\|_{L^2(\gamma)}  + \frac{\mu_\EF\left(\lambda_1^2 -\lambda_2^2\right)}{\lambda_1 d} \| V_{\vec{n}}\|_{L^2(\gamma)} \|W_{\vec{n}}\|_{L^2(\gamma)} +   d\mu_\EF \left\|\frac{\partial V_\vec{n}}{\partial \vec{\tau}} \right\|_{L^2(\gamma)} \left\|\frac{\partial W_{\vec{n}}}{\partial \vec\tau }\right\|_{L^2(\gamma)}  \nonumber\hspace{+2ex}\\
     +    d  \mu_\EF \left\| \frac{\partial V_\vec{\tau}}{\partial \vec{\tau}}\right\|_{L^2(\gamma)} \left\|\frac{\partial W_{\vec{\tau}}}{\partial \vec{\tau}}\right\|_{L^2(\gamma)}+  \frac{\mu_\EF\left(12 \alpha  d + 12\sqrt{K_\PM}\right)}{d \left( \alpha d+4\sqrt{K_\PM}\right)} \| V_\vec{\tau}\|_{L^2(\gamma)}\| W_{\vec{\tau}} \|_{L^2(\gamma)}\nonumber\hspace{+15ex}\\
    \le C_{\mathcal{A}_\gamma} \| \vec{V}\|_{H_\gamma} \| \vec{W}\|_{H_\gamma}, \qquad C_{\mathcal{A}_\gamma} : = \frac{ d\mu}{k_{\min,\TR}} + \frac{\mu_\EF\left(\lambda_1^2 -\lambda_2^2\right)}{\lambda_1 d}+ 2d\mu_\EF+\frac{\mu_\EF\left(12 \alpha  d + 12\sqrt{K_\PM}\right)}{d \left( \alpha d+4\sqrt{K_\PM}\right)} ,\hspace{+1ex} \label{equ:limit_A_gamma}\\
     |\mathcal{B}_\gamma\left(\vec{W}; P\right)|  \le   d \| P\|_{Z_\gamma} \left\|\frac{\partial W_{\vec{\tau}}}{\partial \vec{\tau}} \right\|_{L^2(\gamma)}+   \frac{ \lambda_1 + \lambda_2}{\lambda_1 } \| P\|_{Z_\gamma}  \| W_{\vec{n}}\|_{L^2(\gamma)}\le C_{\mathcal{B}_\gamma } \| P\|_{Z_\gamma} \| \vec{W}\|_{H_\gamma}, \qquad C_{\mathcal{B}_\gamma } : = d + \frac{ \lambda_1 + \lambda_2}{\lambda_1 },\label{equ:limit_B_gamma}\hspace{+0.5ex}
\end{eqnarray}
and
\begin{eqnarray}
    |\mathcal{E}_\gamma \left(P; \Psi\right)| \le \frac{d}{\mu_\EF \lambda_1} \| P\|_{Z_\gamma} \|\Psi \|_{Z_\gamma}= C_{\mathcal{E}_\gamma} \| P\|_{Z_\gamma} \|\Psi \|_{Z_\gamma}, \qquad C_{\mathcal{E}_\gamma}:=\frac{d}{\mu_\EF \lambda_1}.\hspace{+27ex}\label{equ:limit_E_gamma}
    \end{eqnarray}

In the next step, we bound $f_\FP$ defined in \eqref{equ:WeakUncoupledFlow1} and $f_\gamma$, $g_\gamma$ in~\eqref{equ:weakUncoupledgamma}. Considering the Cauchy--Schwarz inequality and the results~\eqref{equ:traceFFTR}, \eqref{equ:tracePM} from the trace theorem, we obtain for $f_\FP$ the upper bound
\begin{eqnarray}
    \left|f_{\FP} (\vec{V}, P; \vec{w}_\FF, \varphi_\PM)\right|
    %
    %
    \le  \left(\frac{\mu_\EF \left(\lambda_1^2-\lambda_2^2\right)}{\lambda_1 d}  + \frac{ \mu_\EF\left(6 \alpha d + 12\sqrt{K_\PM}\right) }{d\left( \alpha d+4\sqrt{{K}_\PM}\right)}  \right)C_{\gamma_\FF,\FF}\| \vec V \|_{H_\gamma} \|\vec{w}_\FF \|_{H_\FF} 
    + \frac{\lambda_1+\lambda_2 }{\lambda_1} C_{\gamma_\FF, \FF} \|P\|_{Z_\gamma}  \|\vec{w}_\FF \|_{H_\FF}\nonumber\\
     +\frac{d}{\mu_\EF\lambda_1}C_{\gamma_\PM,\PM} \|P\|_{Z_\gamma}\|\varphi_\PM\|_{H_\PM} 
     + \frac{\lambda_1 + \lambda_2}{\lambda_1}C_{\gamma_\PM,\PM}\|\vec V \|_{H_\gamma} \|\varphi_\PM\|_{H_\PM}\hspace{+23ex}\nonumber\\
     \le C_{f_\FP} \left( \|\vec V\|_{H_\gamma} +\|P\|_{Z_\gamma} \right)\left( \|\vec{w}_\FF\|_{H_\FF} + \|\varphi_\PM\|_{H_\PM}\right),\hspace{+40.2ex} \label{equ:limit_f_fp}
\end{eqnarray}
where  
\begin{eqnarray*}
    C_{f_\FP} : = \max \left\{ \left(\frac{\mu_\EF \left(\lambda_1^2-\lambda_2^2\right)}{\lambda_1 d}  + \frac{ \mu_\EF\left(6 \alpha d + 12\sqrt{K_\PM}\right) }{d\left( \alpha d+4\sqrt{{K}_\PM}\right)} \right) C_{\gamma_\FF,\FF},\, \frac{\lambda_1+\lambda_2 }{\lambda_1} C_{\gamma_\FF, \FF},\, \frac{d}{\mu_\EF\lambda_1}C_{\gamma_\PM,\PM},\, \frac{\lambda_1 + \lambda_2}{\lambda_1}C_{\gamma_\PM,\PM}\right\}.
\end{eqnarray*}
In a similar way, we obtain the following bound
\begin{eqnarray}
     \left|f_\gamma( \vec{v}_\FF, p_\PM;\vec{W})\right|+ \left|g_\gamma (\vec{v}_\FF, p_\PM;\Psi) \right|
    %
    \le C_{fg_\gamma}\left(\|\vec{v}_\FF \|_{H_\FF}+\| p_\PM \|_{H_\PM} \right)\left(\| \vec W\|_{H_\gamma} + \| \Psi\|_{Z_\gamma} \right), \qquad C_{fg_\gamma}:=C_{f_\FP}.\label{equ:limit_fg_gamma}
\end{eqnarray}
%
Taking into account the upper bounds of the bilinear operators $\mathcal{A}_\FP$, $f_\FP$, $\mathcal{A}_\gamma$, $\mathcal{B}_\gamma$, $\mathcal{E}_\gamma$, $f_\gamma$, $g_\gamma$ from Eqs.~\eqref{equ:limit_A_fp}--\eqref{equ:limit_fg_gamma}  and inequality~\eqref{equ:quadraticrule}, the continuity of $\mathcal{A}^\RE$ is ensured
\begin{eqnarray*}
    \left|\mathcal{A}^\RE \left( \vec{\zeta}^\RE ; \vec{\theta}^\RE \right)\right| 
      &\le& \frac{ C_{\mathcal{A}^\RE}}{4}\left(\| \vec{v}_\FF\|_{H_\FF} +\| p_\PM\|_{H_\PM} +\| \vec{V}\|_{H_\gamma} +\|P\|_{Z_\gamma} \right)\left(\| \vec{w}_\FF\|_{H_\FF} +\| \varphi_\PM\|_{H_\PM}+\| \vec{W}\|_{H_\gamma}+\|\Psi\|_{Z_\gamma}  \right)\\
     &\le& C_{\mathcal{A}^\RE} \|\vec \zeta^\RE\|_{\mathcal{H}^\RE} \|\vec \theta^\RE \|_{\mathcal{H}^\RE},\qquad
         C_{\mathcal{A}^\RE}:=4\max \left\{ C_{\mathcal{A}_\FP}, \, C_{f_\FP},\, C_{\mathcal{A}_\gamma} , \, C_{\mathcal{E}_\gamma},\, C_{\mathcal{B}_\gamma } 
    \right\}.
\end{eqnarray*}

In the last step, we prove the coercivity of $\mathcal{A}^\RE$ on $\mathit{Kern}(\mathcal{B}^\RE)=\{ \vec{\theta}^\RE \in \mathcal{H}^\RE | \nabla \cdot \vec{w}_\FF = 0 \}$ as follows
\begin{eqnarray}
    \mathcal{A}^\RE (\vec{\zeta}^\RE;\vec{\zeta}^\RE)
    \ge  \mathcal{A}_{\FP}(\vec{v}_\FF, p_\PM; \vec{v}_\FF, p_\PM)\!+\!\mathcal{A}_\gamma\left( \vec{V}; \vec{V}\right) \!+\! \mathcal{E}_\gamma \left(P, P \right)
    -\left|f_{\FP} (\vec{V}, P; \vec{v}_\FF, p_\PM)\!+\! f_\gamma( \vec{v}_\FF, p_\PM;\vec{V}) \!-\! g_\gamma (\vec{v}_\FF, p_\PM;P) \right| . \label{equ:reduceddimcoercivity}
\end{eqnarray}
Application of the ellipticity conditions \eqref{equ:boundedK}, \eqref{equ:boundedbeta} to the bilinear terms $\mathcal{A}_\FP$, $\mathcal{A}_\gamma$, $\mathcal{E}_\gamma$ in~\eqref{equ:reduceddimcoercivity} yields
\begin{eqnarray}
    \mathcal{A}_{\FP} (\vec{v}_\FF, p_\PM; \vec{v}_\FF, p_\PM)+\mathcal{A}_\gamma\left( \vec{V}; \vec{V}\right) + \mathcal{E}_\gamma \left(P, P \right) \ge \mu \| \nabla \vec{v}_\FF \|^2_{L^2(\Omega_\FF)} + \frac{k_{\min,\PM}}{\mu} \|\nabla p_\PM\|^2_{L^2(\Omega_\PM)} 
    + \frac{\mu_\EF\left(\lambda_1^2 -\lambda_2^2\right)}{ \lambda_1 d} \| \vec{v}_\FF \cdot \vec{n} \|^2_{L^2(\gamma)}\nonumber\\
    +\frac{\mu_\EF\left(4 \alpha  d + 12\sqrt{K_\PM}\right)}{d \left( \alpha d+4\sqrt{K_\PM}\right)}\| \vec{v}_\FF \cdot \vec{\tau} \|^2_{L^2(\gamma)} + \frac{\mu \beta_{\min}}{\sqrt{K_\TR}}\| \vec{v}_\FF\|^2_{L^2(\gamma)}
    + \frac{d}{\mu_\EF\lambda_1} \| p_\PM\|^2_{L^2(\gamma)} +\frac{d \mu}{k_{\max, \TR}} \| \vec{V}\|^2_{L^2(\gamma)} \nonumber\hspace{+4.7ex}\\
    +\frac{\mu_\EF\left(\lambda_1^2 - \lambda_2^2\right)}{\lambda_1 d} \|V_\vec{n} \|^2_{L^2(\gamma)} + d\mu_\EF \left\| \frac{\partial \vec{V}}{\partial \vec{\tau}}\right\|^2_{L^2(\gamma)}  
    + \frac{\mu_\EF\left(12 \alpha  d + 12\sqrt{K_\PM}\right)}{d \left( \alpha d+4\sqrt{K_\PM}\right)} \|V_\vec{\tau} \|^2_{L^2(\gamma)} +\frac{d}{\mu_\EF\lambda_1} \| P\|^2_{L^2(\gamma)}.  \label{equ:reduceddimcoercivity1part}
\end{eqnarray}
Applying the Cauchy--Schwarz and the generalised Young,   
$ 2ab \le a^2/\delta + \delta b^2\textnormal{ for } \delta>0$, 
inequalities to the remaining terms in~\eqref{equ:reduceddimcoercivity}, we get
\begin{eqnarray}
    \left| f_{\FP} (\vec{V}, P; \vec{v}_\FF, p_\PM)+f_\gamma( \vec{v}_\FF, p_\PM;\vec{V} ) - g_\gamma  (\vec{v}_\FF, p_\PM;P) \right|
      \le 2 \frac{\mu_\EF\left(\lambda_1^2-\lambda_2^2\right)}{\lambda_1 d} \left|\int_\gamma \left(\vec{v}_\FF \cdot \vec{n}|_{\gamma_\FF}\right)V_{\vec{n}}~\mathrm{d} s \right|\nonumber \hspace{+17ex}\\ 
      +2\frac{d}{\mu_\EF\lambda_1}\left|\int_\gamma \left(p_\PM |_{\gamma_\PM}\right)P~\mathrm{d} s\right|  + 2 \frac{\mu_\EF\left(6 \alpha d +12\sqrt{K_\PM}\right)}{d \left(\alpha d+4\sqrt{K_\PM}\right)} \left|\int_\gamma\left(\vec{v}_\FF \cdot \vec{\tau} |_{\gamma_\FF}\right) V_{\vec{\tau}}~\mathrm{d} s\right| \nonumber\hspace{+21ex}\\
      \le \frac{d}{\mu_\EF\lambda_1} \left(\frac{1}{\delta}\| p_\PM \|_{L^2(\gamma)}^2 + \delta \| P\|^2_{L^2(\gamma)}\right) + \frac{\mu_\EF\left(\lambda_1^2-\lambda_2^2\right)}{\lambda_1 d} \left( \frac{1}{\delta_1}\| \vec{v}_\FF \cdot \vec{n}\|_{L^2(\gamma)}^2+\delta_1 \| V_{\vec{n}}\|^2_{L^2(\gamma)} \right) 
      \nonumber\hspace{+18ex}\\ 
      +  \frac{6 \alpha \mu_\EF }{ \alpha d+4\sqrt{K_\PM}} \left( \frac{1}{\delta_2}\left\|\vec{v}_\FF \cdot \vec{\tau}  \right\|^2_{L^2(\gamma)}+ \delta_2 \left\|V_\vec{\tau}  \right\|^2_{L^2(\gamma)}\right)
      +  \frac{12\mu_\EF\sqrt{K_\PM}}{d \left(\alpha d+4\sqrt{K_\PM}\right)} \left( \frac{1}{\delta_3}\left\|\vec{v}_\FF \cdot \vec{\tau}  \right\|^2_{L^2(\gamma)}+ \delta_3 \left\|V_\vec{\tau}  \right\|_{L^2(\gamma)}^2\right),\label{equ:reduceddimcoercivity2part}
\end{eqnarray}
for $\delta, \delta_1, \delta_2, \delta_3 >0$. We choose $\delta_1=\delta_3=1$ and $\delta_2 =2$ in~\eqref{equ:reduceddimcoercivity2part}. Substituting~\eqref{equ:reduceddimcoercivity1part} and \eqref{equ:reduceddimcoercivity2part} into \eqref{equ:reduceddimcoercivity}, taking into account the trace theorem result~\eqref{equ:tracePM} and the Poincar{\'e} inequality \eqref{equ:auxiliarypoincare}, we obtain
\begin{eqnarray}
    \mathcal{A}^\RE (\vec{\zeta}^\RE;\vec{\zeta}^\RE)\ge\mu \| \nabla \vec{v}_\FF \|^2_{L^2(\Omega_\FF)}+ \frac{k_{\min,\PM}}{\mu} \|\nabla p_\PM\|^2_{L^2(\Omega_\PM)}+
    \frac{d \mu}{k_{\max, \TR}} \| \vec{V}\|^2_{L^2(\gamma)}
    + d\mu_\EF \left\| \frac{\partial \vec{V}}{\partial \vec{\tau}}\right\|^2_{L^2(\gamma)} 
    +(1-\delta)\frac{d}{\mu_\EF\lambda_1} \| P\|^2_{L^2(\gamma)}\nonumber\hspace{+3ex}\\
    + \frac{\mu\beta_{\min}}{\sqrt{K_\TR}} \| \vec{v}_\FF\|^2_{L^2(\gamma)} +\frac{\alpha   \mu_\EF}{ \alpha d+4\sqrt{K_\PM}}\| \vec{v}_\FF \cdot \vec{\tau} \|^2_{L^2(\gamma)} 
    +\frac{d}{\mu_\EF\lambda_1 } \| p_\PM\|^2_{L^2(\gamma)} -\frac{1}{\delta}\frac{d C^2_{\gamma_\PM, \PM}}{\mu_\EF\lambda_1} \| p_\PM\|^2_{H_\PM} \nonumber\hspace{+10ex}\\
    \ge  \frac{\mu}{\tilde{C}_{P,\FF}}\|  \vec{v}_\FF \|^2_{H_\FF} + \left(\frac{k_{\min,\PM}}{\mu\tilde{C}_{P,\PM}}  -\frac{1}{\delta}\frac{d C^2_{\gamma_\PM, \PM}}{\mu_\EF\lambda_1} \right)\| p_\PM\|^2_{H_\PM} + d\min \left\{ \frac{ \mu}{k_{\max, \TR}} ,  \mu_\EF \right\} \| \vec{V}\|^2_{H_\gamma}
    +(1-\delta)\frac{d}{\mu_\EF\lambda_1} \| P\|^2_{Z_\gamma}. \label{equ:reducedmodelcoercivityEnd}
\end{eqnarray}
For $0<\delta<1$ and small enough thickness $0<d<d^*$ such that 
$d^*:=\left(\delta  \mu_\EF\lambda_1 k_{\min,\PM}\right) \Big/ \left(\mu\tilde{C}_{P,\PM}C^2_{\gamma_\PM, \PM}\right)$,
the coercivity of $\mathcal{A}^\RE$ is guaranteed
\begin{eqnarray*}
    \mathcal{A}^\RE(\vec{\zeta}^\RE; \vec{\zeta}^\RE) \ge     C^*_{\mathcal{A}^\RE}\| \vec{\zeta}^\RE\|^2_{\mathcal{H}^\RE},\qquad C^*_{\mathcal{A}^\RE}: =\min \left\{  \frac{\mu}{\tilde{C}_{P,\FF}}, \, 
    \left(\frac{k_{\min,\PM}}{\mu\tilde{C}_{P,\PM}}  -\frac{1}{\delta}\frac{d C^2_{\gamma_\PM, \PM}}{\mu_\EF\lambda_1} \right), \, \frac{ d\mu}{k_{\max, \TR}} ,  d\mu_\EF  , \, \frac{(1-\delta)d}{\mu_\EF\lambda_1}  \right\}>0.
\end{eqnarray*}
This completes the proof. 
\end{proof} 

\section{Numerical simulation results}
\label{sec:numresults}
In this section, we report numerical simulation results for the proposed full- and reduced-dimensional models and validate both models against the analytical solutions. \Rthree{Additionally, we perform a comparative analysis of the two proposed models in the context of an industrial filtration problem.} In the full-dimensional case, the models in all three flow regions are discretised by the second-order finite volume method (MAC scheme). We consider uniform rectangular grids that conform on the top and on the bottom of the transition zone. For the reduced-dimensional model, we keep the same discretisations in the free-flow and porous-medium subdomains, and apply the second-order finite difference method at the interface $\gamma$. The two flow problems are solved monolithically using our in-house implementation.

The coupled flow domain is set as $\Omega = [0,L_x]\times[0,L_y]$ with interfaces $\gamma_\FF = \left(0,L_x\right)\times \{y_{\gamma_\FF}\}$ and $\gamma_\PM = \left(0,L_x\right)\times \{y_{\gamma_\PM}\}$. 
The unit tangential and normal vectors are $\vec{\tau}=\vec{e}_1$ and $\vec{n} = \vec{e}_2$ (Fig.~\ref{fig:Fig1intro}). The following physical parameters are chosen $\mu= 1$, $\mu_\EF=1$, $\alpha = 0.1$ and $\vec{\beta}=\vec{0}$. We consider isotropic, homogeneous transition zone and porous medium with the permeability tensors $\ten{K}_\TR = \ten{K}_\PM = 10^{-2}\ten{I}$, and hence the permeability references  $K_\TR=K_\PM=10^{-2}$.


\subsection{Full-dimensional model}
In the first test case, we conduct numerical convergence study of the full-dimensional model \eqref{equ:stokes1}--\eqref{equ:TRBC}, \eqref{equ:darcyupdated}--\eqref{equ:BJS-ICPM}. We choose the analytical solution as follows
\begin{eqnarray}
        u_\FF=\cos{ \left(x_1 \right) } \exp{( x_2-y_{\gamma_\PM})},\hspace{+2.4ex} v_\FF =\sin{\left( x_1\right)} \exp{(x_2-y_{\gamma_\PM})}, \hspace{+2ex}   p_\FF\, =\, \sin{( x_1 + x_2 - y_{\gamma_\PM})},\nonumber\hspace{+3.7ex}\\
        u_\TR = \cos{ (x_1 ) } \exp{( x_2-y_{\gamma_\PM})},\hspace{+2.4ex} 
        v_\TR = \sin{( x_1)} \exp{(x_2-y_{\gamma_\PM})}, \hspace{+2.2ex} 
        p_\TR \,=\,\sin{(x_1 + x_2 - y_{\gamma_\PM})},\hspace{+3.6ex}\label{equ:exactsolution}\\
        p_\PM =-100(x_2-y_{\gamma_\PM})\sin{( x_1)},\nonumber
\end{eqnarray}
where $\vec{v}_\FF= \left(u_\FF, v_\FF\right)^\top $ and $\vec{v}_\TR=\left(u_\TR, v_\TR\right)^\top$. Note that Eqs.~\eqref{equ:exactsolution} satisfy the mass conservation equations~\eqref{equ:stokes1}, \eqref{equ:Brinkman1} and the interface conditions \eqref{equ:continuityICFF}--\eqref{equ:BJS-ICPM}.
\begin{figure}[h!]
    \centering
    \includegraphics[scale = 0.75]{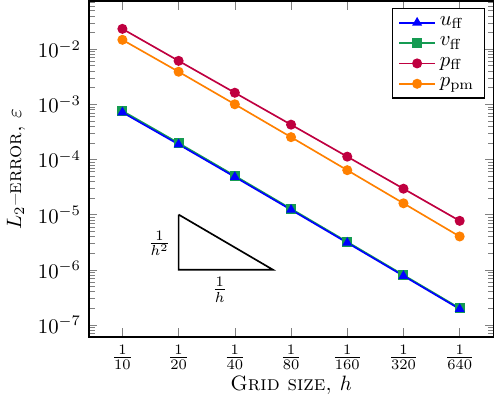}\hspace{4ex} \includegraphics[scale=0.75]{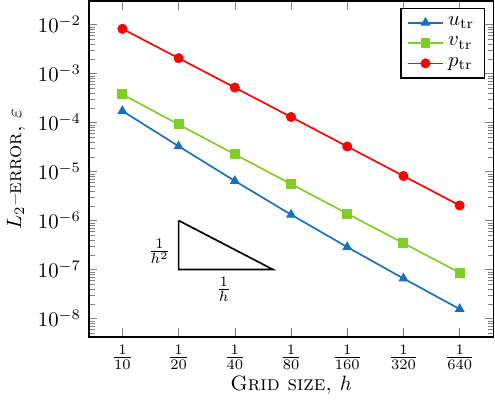}
    \caption{Second-order convergence for the full-dimensional model}
    \label{fig:convergenceanalysisfull}
\end{figure}

The size of the computation domain is $L_x = 1$ and $L_y=2$, where the positions of the interfaces are $y_{\gamma_\PM}=0.9$ and $y_{\gamma_\FF}=1.1$ leading to the transition region thickness $d=0.2$. We choose the Dirichlet boundary conditions on the outer boundaries $\Gamma_\FF$, $\Gamma_\TR$ and $\Gamma_\PM$ (Fig.~\ref{fig:Fig1intro}, left). 
Substituting the exact solution \eqref{equ:exactsolution} with the chosen parameters into Eqs.~\eqref{equ:stokes2}, \eqref{equ:FFBC}, \eqref{equ:Brinkman2}, \eqref{equ:TRBC}, \eqref{equ:darcyupdated} and \eqref{equ:PMBC}, we obtain the source terms and the corresponding boundary conditions.

For the convergence analysis, we compare the numerical simulation results against the analytical solution for all primary variables
\begin{equation}
    \varepsilon_f = \| f - f_h\|_{L^2(\Omega)}, \quad f \in \{u_\FF, v_\FF,p_\FF,  u_\TR, v_\TR, p_\TR, p_\PM\}, \label{equ:L2errorFull}
\end{equation}
where $f$ and $f_h$ denote the analytical and numerical solutions, respectively. 
We report numerical results in Fig.~\ref{fig:convergenceanalysisfull}. They confirm convergence of the discretisation scheme with the second order in all three regions.

\subsection{Reduced-dimensional model}
For the reduced-dimensional model \eqref{equ:stokes1}--\eqref{equ:FFBC}, \eqref{equ:darcyupdated}, \eqref{equ:PMBC}, \eqref{equ:averagedMass}, \eqref{equ:averagedMomentumnormal}, \eqref{equ:averagedMomentumtangentialupdate}, \eqref{equ:averageDirichlet}, \eqref{equ:averageNeuman} with the transmission conditions \eqref{equ:transmissionFFnormal}--\eqref{equ:transmissionPM}, we choose the same analytical solution \eqref{equ:exactsolution} in the free-flow and porous-medium subdomains. The exact solution $\left(U,V,P\right)$ on $\gamma$ is obtained by averaging the exact solution $(u_\TR,v_\TR, p_\TR)$ given in~\eqref{equ:exactsolution} across the transition zone
\begin{eqnarray}
    U = \frac{\cos{ \left(x_1 \right) }}{d} \left(\exp{(d)}-1\right), \quad
    V = \frac{\sin{ \left(x_1 \right) }}{d} \left(\exp{(d)}-1\right),\quad 
    P = -\frac{1}{d}\left(\cos{\left(x_1 + d\right)}-  \cos{\left(x_1 \right)}\right).
\end{eqnarray}
The geometry of the computational domain is $L_x = 1$, $L_y=1+d$, $\, y_{\gamma_\PM}=0.5$ and $y_{\gamma_\FF}=0.5+d$, where a narrow transition zone with $d=5\cdot10^{-4}$ is considered. We choose $\lambda_1=4$ and $\lambda_2=2$, which refer to the quadratic profile of the normal velocity across the transition region (Table~\ref{tab:closureconditionnormal}). The Neumann conditions are imposed on $\partial \gamma$ as well as on the left and right sides of $\Gamma_\FF$~(Fig~\ref{fig:Fig1intro}, right). The Dirichlet conditions are set on the remaining outer boundaries. We obtain the averaged source terms $F_{\vec{n}}$, $F_{\vec{\tau}}$ in \eqref{equ:averagedMomentumnormal} and \eqref{equ:averagedMomentumtangentialupdate} by averaging $\vec{f}_\TR$ from Eq.~\eqref{equ:Brinkman2} across the transition region.

In addition to the $L_2$-errors for $u_\FF$, $v_\FF$, $p_\FF$ and $p_\PM$ as in \eqref{equ:L2errorFull},  we compute the errors for the averaged velocity and pressure
\begin{equation}
    \varepsilon_f = \| f - f_h\|_{L^2(\gamma)}, \quad f \in \{U, V, P\},
\end{equation}
where $f$ is the averaged analytical solution and $f_h$ denotes the numerical solution on the complex interface $\gamma$.
The results presented in Fig.~\ref{fig:convergenceanalysisreduced} confirm second-order convergence of the discretisation scheme for the reduced-dimensional model as well, and validate the proposed model with the derived transmission conditions.
\begin{figure}[h!]
    \centering
    \includegraphics[scale = 0.75]{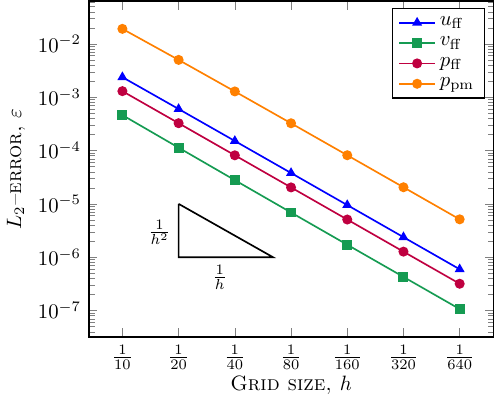}\hspace{4ex} \includegraphics[scale=0.75]{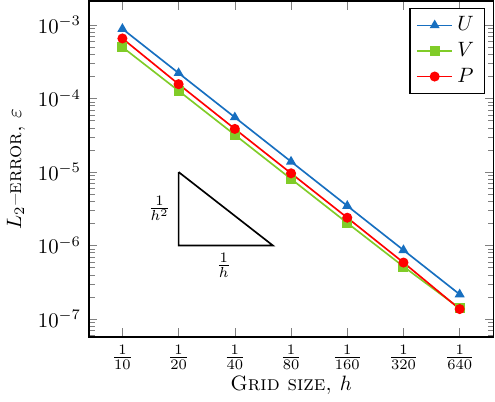}
    \caption{Second-order convergence for the reduced-dimensional model}
    \label{fig:convergenceanalysisreduced}
\end{figure}

\Rthree{\subsection{Filtration problem}
In the last test case, we investigate the proposed full- and hybrid-dimensional Stokes--Brinkman--Darcy models based on a typical industrial filtration problem. Here, the fluid enters from the bottom of the porous medium, passes through the transition region, and exits through the free-flow domain. In this problem, the flow is arbitrary near the fluid--porous interface. The size of the flow domain is $L_x= 1$ and $L_y = 1.005$, where the positions of the top and the bottom of the transition region are $y_{\gamma_\PM}=0.5$ and $y_{\gamma_\FF} = 0.505$ and the thickness is $d=0.005$. The grid size is fixed $h=h_x = h_y =1/800$. In contrast to the previous test cases with analytical solutions, we consider here relevant physical parameters: $\mu = \mu_\EF = 10^{-3}$, $\alpha = 1$, $\vec{\beta}=\vec{0}$, $\ten{K}_\TR= 10^{-3}\ten{I}$, and $\ten{K}_\PM=10^{-8}\ten{I}$. The source terms are $\vec{f}_\FF=\vec{0}$, $\vec{f}_\TR=\vec{0}$ and $q=0$.

The boundary conditions in  this test case are schematically presented in  Fig.~\ref{fig:FilteredProblem} (left). On the ``wall" boundaries $\Gamma_{\FF, \mathrm{wall}}= \left(\{0\}\times [0.505, 1.005]\right)\cup\left( [0,1] \times \{1.005\}\right)$,  $\Gamma_{\TR, \mathrm{wall}}= \left(\{0\}\times [0.5, 0.505]\right)\cup \left(\{1\}\times [0.5, 0.505]\right)$, $\partial \gamma_\mathrm{wall} = \{(0,0.5025), (1,0.5025)\}$ and $\Gamma_{\PM, \mathrm{wall}}= \left( \{0\} \times [0,0.5]\right) \cup \left( [0,0.25] \times \{0\}\right)\cup \left( [0.75,1]\times \{0\}\right) \cup\left(\{1\} \times [0,0.5] \right) $, we consider the no-slip boundary conditions 
\begin{equation}
    \vec{v}_\FF = \vec{0} \quad \textnormal{on } \Gamma_{\FF, \mathrm{wall}}, \quad  \vec{v}_\TR = \vec{0} \quad \textnormal{on } \Gamma_{\TR, \mathrm{wall}}, \quad \vec{V} = \vec{0} \quad \textnormal{on } \partial \gamma_\mathrm{wall},  \quad \textnormal{and} \quad \vec{v}_\PM \cdot \vec{n}_\PM = 0 \quad \textnormal{on } \Gamma_{\PM, \mathrm{wall}}. 
\end{equation}
On the bottom of the porous-medium region, a parabolic ``inflow" boundary conditions are taken into account 
\begin{equation}
    \vec{v}_\PM \cdot \vec{n}_\PM= -0.1(x_1 - 0.25)(x_1-0.75) \quad \textnormal{on } \Gamma_{\PM, \mathrm{in}} = [0.25, 0.75] \times \{0\}. 
\end{equation}
On the right external boundary of the free-flow region, the ``do-nothing" boundary conditions 
\begin{equation}
    \ten{T}(\vec{v}_\FF, p_\FF) \cdot \vec{n}_\FF = \vec{0}\quad \textnormal{on }\Gamma_{\FF, \mathrm{dn}} = \{1\} \times [0.505, 1.005], 
\end{equation}
are imposed, allowing the flow to exit the domain naturally. The velocity magnitude and streamlines for the reduced-dimensional model are shown in Fig.~\ref{fig:FilteredProblem} (right). The simulation results for the full-dimensional model are visually nearly identical, therefore, we omit the corresponding plot and focus on the quantitative differences instead.
\begin{figure}[h!]
    \centering
    \includegraphics[scale =1]{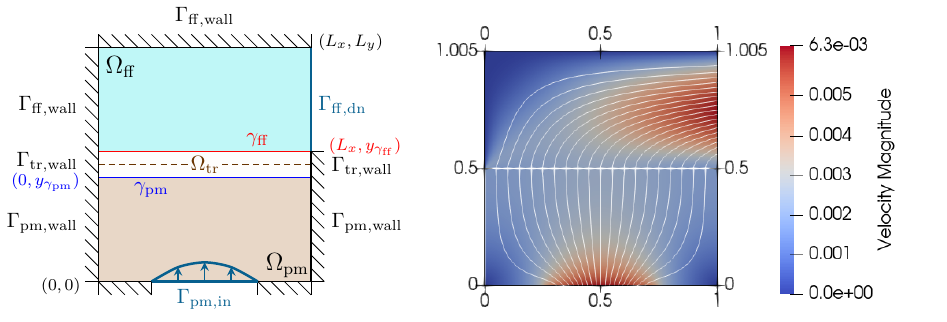}
    \caption{Geometric setting of the filtration problem (left) and velocities for the reduced-dimensional model (right)}
  \label{fig:FilteredProblem}
\end{figure}

To compare the full- and reduced-dimensional models, we first compute the averaged velocities and pressure of the full-dimensional model across the transition region
\begin{equation}
u^\AV_\TR :=  \frac{1}{d}  \int^{y_{\gamma_\FF}}_{y_{\gamma_\PM}} u_\TR \,\mathrm{d} x_2, \quad  v^\AV_\TR :=  \frac{1}{d}  \int^{y_{\gamma_\FF}}_{y_{\gamma_\PM}} v_\TR \,\mathrm{d} x_2, \quad  p^\AV_\TR :=  \frac{1}{d}  \int^{y_{\gamma_\FF}}_{y_{\gamma_\PM}} p_\TR \,\mathrm{d} x_2.
\end{equation}
Then, we compare these values with the solutions $U, V, P$ of the reduced-dimensional model on the complex interface~$\gamma$ as follows
\begin{equation}
    \varepsilon_{u}^r = \frac{\| u^\AV_\TR  - U  \|_{L^2(\gamma)}}{\|u^\AV_\TR\|_{L^2(\gamma)}},   \quad  \varepsilon_{v}^r = \frac{\| v^\AV_\TR  - V \|_{L^2(\gamma)}}{\|v^\AV_\TR\|_{L^2(\gamma)}}, \quad \textnormal{and}\quad    \varepsilon_{p}^r = \frac{\| p^\AV_\TR  - P  \|_{L^2(\gamma)}}{\|p^\AV_\TR\|_{L^2(\gamma)}}.
\end{equation}
The relative deviations are presented in Table~\ref{tab:relativeDEV} for different choices of the normal velocity profile (Table~\ref{tab:closureconditionnormal}). The numerical simulation results with quadratic normal velocity profile yield the smallest relative deviations. 
In addition, we provide the comparison of the normal and tangential velocity profiles on $\gamma_\FF$ ($x_2=0.505$) for both models in Fig.~\ref{fig:velocitycptfullreduced}. Here, the quadratic assumption for the normal velocity profile ($\lambda_1=4$, $\lambda_2 = 2$) is applied for the reduced-dimensional model.

\begin{table}[h]
\begin{center}
\begin{tabular}{r@{\qquad}r@{\quad}r@{\quad}r@{\quad}r@{\quad}r@{\quad}r@{\quad}r}
\hline
\multicolumn{1}{c}{\rule{0pt}{12pt}Profiles}&\multicolumn{1}{c}{$\varepsilon^r_u$}&\multicolumn{1}{c}{$\varepsilon^r_v$}&\multicolumn{1}{c}{$\varepsilon^r_p$}\\[2pt]
\hline\rule{0pt}{12pt}
Linear $(\lambda_1 = 2, \lambda_2 =0)$  & 8.0680e-2 & 3.1463e-2 & 1.3016e-1 \\
Piecewise linear  $(\lambda_1 = 3, \lambda_2 =1)$ &3.5585e-2 & 2.2407e-2 & 9.1480e-2 \\
Quadratic  $(\lambda_1 = 4, \lambda_2 =2)$  & 2.8920e-2 & 1.8101e-2 &7.1120e-2\\
\hline
\end{tabular}
\end{center}
\caption{Relative derivations for different normal velocity profiles}\label{tab:relativeDEV}
\end{table}

\begin{figure}[h!]
    \centering
    \includegraphics[scale = 0.75]{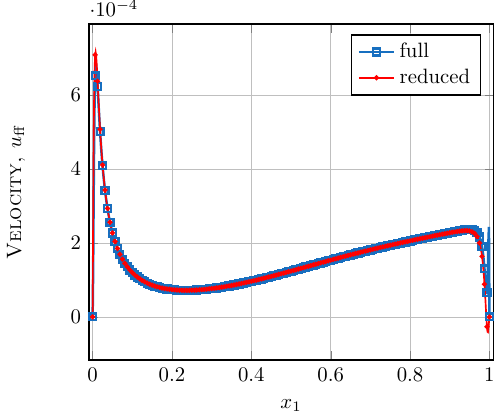}\hspace{4ex} \includegraphics[scale=0.75]{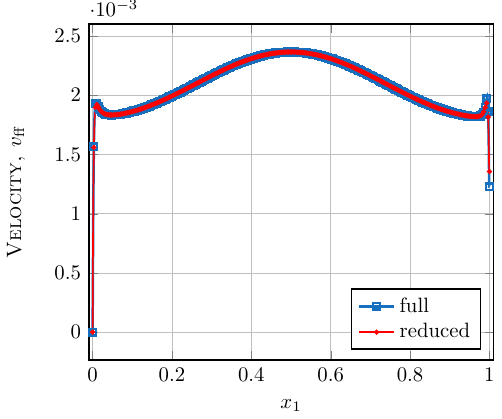}
    \caption{Normal (left) and tangential (right) velocity profiles at $x_2= 0.505$ for the coupled full-dimensional and dimensionally reduced models}
    \label{fig:velocitycptfullreduced}
\end{figure}
}
\Rthree{Although the numerical simulation results obtained from the full-dimensional and reduced-dimensional models are very similar, their computational times differ significantly. CPU time for the full-dimensional model is 359.58~[s], compared to 228.18~[s] for the reduced-dimensional model. This test case demonstrates 36.55\% decrease in CPU time when using the reduced-dimensional model. The computations are carried out on a personal computer with a 12\textsuperscript{th} Gen Intel\textsuperscript{\textregistered} Core\texttrademark{} i7-1255U processor and 32.0\,GB RAM using our in-house C++ code.}

\section{Conclusions}
\label{sec:conclusion}
In this work, we propose the full-dimensional Stokes--Brinkman--Darcy model and derive its dimensionally reduced counterpart to describe fluid flows in joint plain-fluid and porous-medium regions. The full-dimensional model is composed of the equi-dimensional Stokes equations, the Brinkman equations and Darcy's law in the respective flow domains. To couple the equations in the transition zone and in the free flow, continuity of velocity and  stress jump conditions are imposed on the top of the transition region. On the bottom of the transition zone, we apply the classical set of coupling conditions using the Beavers--Joseph--Saffman concept.

Due to multiphysics nature of the coupled flows, the full-dimensional model requires fine grids in the transition zone.  Therefore, we derive the dimensionally reduced model, which is an effective approximation of the proposed full-dimensional model. In the reduced-dimensional formulation, the transition region thickness is significantly smaller in comparison to the length of the entire flow domain. Thus, we represent the transition region as a complex interface, which is able to store and transport mass and momentum. We average the Brinkman equations across the transition region that yields the model of co-dimension one, and consider the appropriate transmission conditions to get the closed model formulation. The transmission conditions are derived using \emph{a priori} hypotheses on the pressure and velocity profiles in the transition zone. A constant pressure profile and a quadratic tangential velocity profile are applied in this work. For the normal velocity component, we consider three cases: a linear, a piecewise linear and a quadratic profile.

We conduct the well-posedness analysis for both models by verifying the Babu\v{s}ka--Brezzi conditions and prove existence and uniqueness of weak solutions for the two Stokes--Brinkman--Darcy problems. In the hybrid-dimensional case, there is a restriction on the transition region thickness depending on the physical parameters. Some numerical simulations for the proposed models are reported. They demonstrate the second-order convergence of the chosen discretisation schemes and indicate that the models in the transition zone are capable to describe fluid flow between the free-flow and porous-medium domains.
\Rthree{Furthermore, we conduct a comparison study of the full-dimensional and reduced-dimensional models proposed in this work in the context of industrial filtration.} \Rone{ Comparison study} of the dimensionally reduced model \Rone{ with} alternative interface conditions available in the literature is \Rone{ presented in our recent  work~\cite{Ruan_Rybak_24}. In that study, we validated the hybrid-dimensional Stokes–Brinkman–Darcy model proposed in this work against the pore scale resolved model considering multiple test cases. Additionally, we compared numerical simulation results obtained using the hybrid-dimensional model with those from the Stokes--Darcy model employing both classical and generalised interface conditions. The proposed hybrid-dimensional Stokes--Brinkman--Darcy model demonstrated its applicability to describe arbitrary flows in coupled systems, showed its advantages in comparison to the classical coupling conditions and slightly more accurate results in comparison to the generalised conditions.} Further extension of this study involves the Navier--Stokes--Brinkman--Darcy models to handle multi-dimensional inertial flows in fluid--porous systems.

\appendix

\section{Auxiliary results}
\renewcommand{\thesection}{\Alph{section}}
\label{sec:appB}
We introduce some useful inequities as well as the well-posedness theorem~\cite[Cor.~I.4.1]{girault2012finite}, the trace theorem~\cite[Thm.~I.1.5]{girault2012finite} and the Poincar{\'e} theorem~\cite[Thm.~I.1.1]{girault2012finite} used in the analysis of the proposed coupled models. 

The following inequalities hold true
\begin{eqnarray}
    (a+b)^2 \le 2 (a^2 + b^2),\qquad
    (a+b+c+d)^2 \le  4 \left( a^2 + b^2 +c^2+d^2\right), \quad \forall a, b, c ,d \in \mathbb{R}^+_0.\label{equ:quadraticrule}
\end{eqnarray}

\begin{appendixtheorem}[Well-posedness] \label{thm:wellposedness}Let $\mathcal{H}$ and $\mathcal{Z}$ be two Hibert spaces with the norms $\|\cdot\|_{\mathcal{H}}$ and $\|\cdot\|_{\mathcal{Z}}$, respectively. 
We consider the following variational problem for the bilinear operators $\mathcal{A}:\mathcal{H} \times \mathcal{H} \to \mathbb{R}$, $\mathcal{B}: \mathcal{H}\times \mathcal{Z} \to \mathbb{R}$ and the linear functional $\mathcal{L}: \mathcal{H}\to \mathbb{R}:$ \par
    Find $\vec \zeta \in \mathcal{H}$ and $\vec \xi \in \mathcal{Z}$ such that 
    \begin{subequations}
    \begin{eqnarray}
        \mathcal{A}(\vec{\zeta}; \vec{\theta}) + \mathcal{B}(\vec{\theta}; \vec \xi ) &=& \mathcal{L} \left(\vec{\theta}\right),    \hspace{+1ex} \forall \vec{\theta} \in \mathcal{H}, \label{equ:thmwellposed1}   \\
        \mathcal{B} (\vec{\zeta}; \vec \psi) &=& 0, \hspace{+5ex}  \forall \vec \psi \in \mathcal{Z}.     \label{equ:thmwellposed2} 
\end{eqnarray}\label{equ:thmwellposed} 
    \end{subequations}
Problem \eqref{equ:thmwellposed} is well-posed if and only if the following 
conditions are fulfilled:
\begin{itemize}
    \item $\mathcal{A}$ and $\mathcal{B}$ are continuous:
    \begin{eqnarray}
        | \mathcal{A}(\vec{\zeta}; \vec{\theta})| \le C_{\mathcal{A}}\| \vec{\zeta}\|_{\mathcal{H}} \|  \vec{\theta}\|_{\mathcal{H}},& C_{\mathcal{A}}>0, &  \forall \vec{\zeta},  \vec{\theta}\in \mathcal{H}, \label{continuous1}\\ 
        | \mathcal{B}(\vec{\zeta}; \vec \psi) |\le C_{\mathcal{B}}\| \vec{\zeta}\|_{\mathcal{H}} \| \vec \psi\|_\mathcal{Z},& C_{\mathcal{B}}>0, & \forall   \hspace{+1ex}\vec{\zeta}  \in \mathcal{H}, \,\hspace{+1ex}  \vec\psi \in \mathcal{Z}, \label{continuous2}
    \end{eqnarray}
    \item $\mathcal{A}$ is coercive on $\mathit{Kern}(\mathcal{B})$:
    \begin{eqnarray}
        \mathcal{A}(\vec{\zeta}; \vec{\zeta}) \ge C^*_{\mathcal{A}}\| \vec{\zeta}\|_{\mathcal{H}}^2, \quad C^*_{\mathcal{A}}>0, \quad \forall \vec{\zeta} \in \mathcal{H}, \label{coecitive}
    \end{eqnarray}
    \item $\mathcal{B}$ is inf-sup-stable:
    \begin{equation}
        \inf_{\vec \psi\in \mathcal{Z}} \sup_{\vec{\zeta}\in \mathcal{H}} \frac{ \mathcal{B}(\vec{\zeta}; \vec \psi) }{\| \vec{\zeta} \|_{\mathcal{H}} \| \vec\psi \|_\mathcal{Z}} \ge C^*_{\mathcal{B}}, \quad C^*_{\mathcal{B}}>0. \label{inf-sup}
    \end{equation}
\end{itemize}
\end{appendixtheorem}
The inf-sup condition \eqref{inf-sup} is also called the Babu\v{s}ka--Brezzi condition~\cite{babuvska1973finite, brezzi1974existence}.


\begin{appendixtheorem}
[Trace theorem]\label{thm:tracetheorem}
Let $\Omega \subset \mathbb{R}^n$ be a bounded Lipschitz domain with boundary $\Gamma$. Then the mapping $f\to f|_{\Gamma}$  has a unique linear continuous extension as an operator from 
    $H^1(\Omega)$ onto $H^{1/2}(\Gamma).$
\end{appendixtheorem}
Note that Theorem~\ref{thm:tracetheorem} implies
\begin{eqnarray}
        \exists C_{\Gamma}>0 &\textnormal{s.t.}& \| f|_{\Gamma} \|_{L^2(\Gamma)} \le \|f |_{\Gamma}\|_{H^{1 / 2}(\Gamma)} \le C_{\Gamma} \|f \|_{H^1 (\Omega)}.\label{equ:tracegeneral}
\end{eqnarray}

\begin{appendixtheorem}[Poincar{\'e} theorem] \label{thm:poicare} Let the domain $\Omega_i$ for $i\in \{\FF,\TR,\PM\}$ be connected and bounded. Then there exists $C_{P,i}=C_{P,i} (\Omega_i)>0$ such that 
\begin{eqnarray}
    \| f \|_{L^2(\Omega_i)} \le  C_{P,i} \| \nabla f \|_{L^2(\Omega_i)}, \quad \forall f \in H^1_0(\Omega_i).
\end{eqnarray}
\end{appendixtheorem}

Taking $\|f\|_{H^1(\Omega_i)} ^2 = \|f\|_{L^2(\Omega_i)} ^2 +\|\nabla f\|_{L^2(\Omega_i)} ^2 $ into account, we get
\begin{eqnarray}
    \|f\|_{H^1(\Omega_i)} ^2 \le \tilde{C}_{P,i} \|\nabla f\|_{L^2(\Omega_i)} ^2, \quad \tilde{C}_{P,i}=1+C_{P,i}^2>1. \label{equ:auxiliarypoincare}
\end{eqnarray}

\renewcommand{\thesection}{\appendixname\ \Alph{section}}

\section*{Acknowledgement}
\vspace{-2ex}
The work is funded by the Deutsche Forschungsgemeinschaft (DFG, German Research Foundation) -- Project Number 490872182 and Project Number 327154368~-- SFB 1313.
\vspace{-3ex}
\section*{Data availability}
\vspace{-2ex}
The datasets generated and analysed during the current study are 
available from the corresponding author on reasonable request.
\vspace{-3ex}
\section*{Declaration of interest}
\vspace{-2ex}
The authors have no competing interests to declare that are relevant to the content of this paper.
\vspace{-3ex}
\section*{Declaration of generative AI and AI-assisted technologies in the writing process}
\vspace{-2ex}
During the preparation of this work, the authors did not use any generative AI and AI-assisted technologies.

\bibliographystyle{elsarticle-num} 
\bibliography{FLUPOR}


\end{document}